\documentclass[final,3p,11pt]{elsarticle}

\usepackage{verbatim}
\usepackage{mathrsfs}
\usepackage{amsfonts}
\usepackage{amsmath}
\usepackage{amssymb}
\usepackage{bbm}
\usepackage{mathrsfs}
\usepackage{graphicx}
\usepackage[usenames,dvipsnames]{xcolor}
\usepackage{enumerate}
\usepackage{theorem}

\allowdisplaybreaks

\newcommand{\ea}{\end{array}}
\newtheorem{theorem}{Theorem}[section]
\newtheorem{proposition}{Proposition}[section]
\newtheorem{corollary}{Corollary}[section]
\newtheorem{lemma}{Lemma}[section]
\newtheorem{definition}{Definition}[section]

\newtheorem{Cod}{Condition}[section]

\begin{document}

\begin{frontmatter}

\title{Large deviations for  2D Stochastic Chemotaxis-Navier-Stokes System}

\author[mymainaddress]{Yunfeng Chen}
\ead{cyf01@mail.ustc.edu.cn}

\author[mysecondaryaddress,myfourthaddress]{Xuhui Peng}
\ead{xhpeng@hunnu.edu.cn}

\author[mythirdaryaddress]{Jianliang Zhai}
\ead{zhaijl@ustc.edu.cn}

\address[mymainaddress]{School of Mathematics, University of Science and Technology of China, Hefei, 230026, China}
\address[mysecondaryaddress]{MOE-LCSM, School of Mathematics and Statistics, Hunan Normal University,
Changsha, Hunan 410081, P.R. China;}
\address[mythirdaryaddress]{School of Mathematical Sciences and Wu Wen-Tsun Key Laboratory of Mathematics, University of Science and Technology of China, Hefei, 230026, China;}
\address[myfourthaddress]{Key Laboratory of Applied Statistics and Data Science, Hunan Normal University, College of Hunan Province, Changsha 410081, P.R.China}

\begin{abstract}
In this paper, we establish a large deviation principle for 2D stochastic Chemotaxis-Navier-Stokes equation perturbed by a small multiplicative
noise. The main difficulties come from the lack of a suitable compact embedding into the space occupied by the solutions and the inherent complexity of equation. Finite dimensional projection arguments and introducing suitable stopping times play important roles.



\end{abstract}


\end{frontmatter}


\section{Introduction}
\setcounter{equation}{0}
Let  $\mathcal{O}\subset\mathbb{R}^2$ be a bounded convex domain with smooth boundary
$\partial \mathcal{O}$.
The purpose of this paper is to establish a Freidlin-Wentzell type large deviation principle (LDP) for the solutions of the coupled 2D stochastic Chemotaxis-Navier-Stokes system on $\mathcal{O}$:
\begin{align}\label{eq system 00}
& dn^\varepsilon+u^\varepsilon\cdot \nabla n^\varepsilon dt=\delta \Delta n^\varepsilon dt-\nabla\cdot(\chi(c^\varepsilon)n^\varepsilon\nabla c^\varepsilon)dt,\nonumber\\
& dc^\varepsilon+u^\varepsilon\cdot \nabla c^\varepsilon dt=\mu \Delta c^\varepsilon dt-k(c^\varepsilon)n^\varepsilon dt,\nonumber\\
& du^\varepsilon+(u^\varepsilon\cdot \nabla)u^\varepsilon dt+\nabla P^\varepsilon dt=\nu \Delta u^\varepsilon dt-n^\varepsilon\nabla\phi ~dt+\varepsilon\sigma(u^\varepsilon )dW_t,\\
& \nabla\cdot u^\varepsilon=0,\ \ \ \ \ t>0,\ x\in\mathcal{O},\nonumber
\end{align}
as the  small parameter $\varepsilon>0$ converges to $0$, with the boundary conditions
\begin{align}\label{eq boundary condition 1}
\frac{\partial n^\varepsilon}{\partial v}=\frac{\partial c^\varepsilon}{\partial v}=0\text{  and  }u^\varepsilon=0\ \ \text{   for }x\in\partial\mathcal{O}\text{ and }t>0,
\end{align}
and the initial conditions
\begin{align}\label{eq boundary condition 2}
n^\varepsilon(0,x)=n_0(x),\ \ \ c^\varepsilon(0,x)=c_0(x),\ \ \ u^\varepsilon(0,x)=u_0(x),\ \ \ x\in\mathcal{O}.
\end{align}
Here $v$ stands for the inward normal on the boundary $\partial O$.

The deterministic version ($\varepsilon=0$) of equation (\ref{eq system 00}), proposed by Tuval $ et\ al.$ \cite{TCDWKG 2005}, has attracted significant attention from mathematicians in recent years, as evidenced by a multitude of studies(e.g., \cite{Cao 2016,CKL 2014,Duan Lorz Markowich 2010,Ishida 2015,KMS 2016,Li Li 2016,Liu-Lorz,Tao Winkler 2013,Winkler,Winkler 10,Winkler 2016,Winkler 2015 01,Zhang Zheng 2014}). The model describes the behavior of aerobic bacterial populations in sessile water droplets \cite{DL, TCDWKG 2005}, capturing three primary phenomena: the chemotactic movement of bacteria toward oxygen, the impact of dense bacteria on the fluid motion due to gravity, and the convective transport of both cells and oxygen. Further biological context is elaborated in \cite{Lorz 2010,TCDWKG 2005}. Here, cell density and oxygen concentration are expressed as $n^\varepsilon = n^\varepsilon(x, t)$ and $c^\varepsilon = c^\varepsilon(x, t)$, respectively, with fluid velocity and pressure denoted as $u^\varepsilon = u^\varepsilon(x, t)$ and $P^\varepsilon = P^\varepsilon(x, t)$. The diffusion coefficients for cells, chemicals, and fluid are represented by the positive constants $\delta, \mu$ and $\nu$, respectively. The spatiotemporal variable $\phi=\phi(x)$ defines the gravitational potential, while $\chi(c^\varepsilon)$ and $k(c^\varepsilon)$ represent the chemotactic sensitivity and oxygen consumption rate, assumed to be sufficiently differentiable functions.
The process $\{W_t,\ t\geqslant0\}$ is a cylindrical Wiener process, signifying external stochastic influences.

Together with Zhang \cite{Zhai Zhang 2020}, the third author recently showed the existence and uniqueness of  strong solutions in both the analytical sense and the probabilistic sense to equation (\ref{eq system 00}). Notably, the strong  analytical solutions identified in \cite{Zhai Zhang 2020} belong to the  space $L^\infty([0,T],C^0(\bar{\mathcal{O}}))\times L^\infty([0,T],W^{1,q}(\mathcal{O}))\times L^\infty([0,T],D(A^\alpha))$, which does not constitute a Polish space. There seems few results on  LDP for stochastic partial differential equations (SPDEs) taking values in non-Polish spaces, and such issues are challenging due to the lack of space separability.  In this paper, we demonstrate that these solutions are situated within the Polish space $ C([0,T],C^0(\bar{\mathcal{O}}))\times C([0,T],C^0(\bar{\mathcal{O}})) \times C([0,T],D(A^\alpha))$ (see Proposition \ref{regularity}), and we will prove the LDP  of equation (\ref{eq system 00}) in this space.

To formulate the LDP, we adopt the weak convergence methodology pioneered by Budhiraja, Dupuis, and Maroulas \cite{BDM}. Matoussi, Sabbagh, and Zhang \cite{MSZ} provided an advanced condition for satisfying the Budhiraja-Dupuis-Maroulas criteria, offering enhanced efficacy for the analysis of large deviations in stochastic evolution equations; see e.g., \cite{DWZZ,WZ}. In this paper we will use this method.

The primary challenge is to demonstrate the strong convergence in the space $ C([0,T],C^0(\bar{\mathcal{O}}))\times C([0,T],C^0(\bar{\mathcal{O}})) \times C([0,T],D(A^\alpha))$  of solutions to the so-called skeleton equations through the weak convergence of controls  $h_n$ in $L^2([0, T],U)$ (see Condition \ref{cod} (a) in the proof of Theorem \ref{Thm main 1}). There are two classical methods for achieving such convergence. One is to  secure a space with compact embedding; see  e.g., \cite{BDM,FG,WZZ}. Another is to resort to a time approximation strategy as explored in \cite{CM,DZ}.
Nonetheless, for our problem, finding an appropriate compact embedding seems to be difficult, and the time approximation method may course a large amount of computation.
Drawing inspiration from \cite{CCY} and \cite{L}, we address this challenge by applying finite-dimensional projection arguments(see Section \ref{a}).
Additionally, introducing suitable stopping times play an important role to overcome   the inherent complexity of equations; see e.g., (\ref{tauM}) and (\ref{tau}).

Finally, it is worth noting that Zhang and Liu \cite{ZL} have recently established the global solvability of three-dimensional stochastic Chemotaxis-Navier-Stokes systems with L\'evy processes. Additionally,  Hausenblas $et\ al$.\cite{HMR} have demonstrated the existence of a unique probabilistic strong solution for two-dimensional stochastic Chemotaxis-Navier-Stokes systems with additional noise in the $c$-equation.

The rest of the paper is organized as follows. In Section \ref{section2}, we introduce the necessary notations and preliminaries, present the hypotheses, and state the main result. The existence and uniqueness of solutions to the associated skeleton equations are proven in Section \ref{skeleton}. Finally, we verify  two conditions to establish LDP in Section \ref{a} and Section \ref{b}, respectively.

\section{Framework and Statement of the Main Results}\label{section2}
\setcounter{equation}{0}
First of all, let us  briefly recall some essential definitions and topological spaces that will be used throughout this paper.

Let $L^q(\mathcal{O}), 1\leqslant q\leqslant \infty$ denotes the $L^q$ space with respect to the Lebesgue measure with norm $\|\cdot\|_{L^q}$. $W^{k,q}(\mathcal{O})$ denotes the Sobolev space of functions whose distributional derivatives of  order up to  $k$ belong to $L^q$ with norm $\|\cdot\|_{k,q}$.
Let $A$ be the realization of the Stokes operator $-\mathcal{P}\Delta$, where $\mathcal{P}$ denotes the Helmholtz projection from $L^2(\mathcal{O})$ into the space $H=\{\varphi\in L^2(\mathcal{O}):\nabla\cdot\varphi=0,\varphi(x)|_{x\in\partial \mathcal{O}}=0\}$. Let $A^\alpha$ be the fractional powers of the nonegative operator $A$ for all $\alpha>0$ with domain {$D(A^\alpha)=\{\varphi\in H:\|\varphi\|_\alpha := \|\varphi\|_{D(A^\alpha)} = \|A^\alpha\varphi\|_{L^2}<\infty\}$.
}
In the sequel, $(e^{t\Delta})_{t\geqslant 0}$, $\left(e^{-tA}\right)_{t\geqslant 0}$ will denote respectively the Neumann heat semigroup and the Stokes semigroup with Dirichlet boundary condition. 
Throughout this article, $C$ denotes a generic constant that may change from
line to line.

In this paper, we study (\ref{eq system 00}) with the following condition on the parameters and functions:
\begin{itemize}
\item[({\bf H.1})] \begin{itemize}
             \item[(a)] $\chi\in C^2([0,\infty)),\ \chi>0$ in $[0,\infty)$,

             \item[(b)] $k\in C^2([0,\infty)),\ k(0)=0,\ k>0$ in $(0,\infty)$,

             \item[(c)] $\phi\in C^2(\bar{\mathcal{O}})$,
           \end{itemize}
\item[({\bf H.2})] $(\frac{k(c)}{\chi(c)})'>0,\ (\frac{k(c)}{\chi(c)})''\leqslant0,\ (\chi(c)\cdot k(c))'\geqslant0$ on $[0,\infty)$.
\end{itemize}

Let $U$ be a separable real  Hilbert space and  $\{W_t,\ t\geqslant0\}$ a $U$-cylindrical  Wiener process
on a given complete, filtered probability space $(\Omega,\mathcal{F},\{\mathcal{F}_t\}_{t\geqslant0}, \mathbb{P})$,
representing the driving  external random force.

 Let ${\mathcal{L}^2}(U,D(A^\beta))$ denote the space of Hilbert-Schmidt operators $g$ from $U$ into $D(A^\beta)$, and its norm is denoted by $\|g\|_{{\mathcal{L}^2_\beta}}$. For a  mapping $\sigma:D(A^\beta)\rightarrow \mathcal{L}^2(U,D(A^\beta))$, we introduce the following assumptions: there exists a positive constant $K$ such that

\begin{itemize}
\item[({\bf H.3})]   for all $u_1,u_2,u\in H$,
           \begin{align*}
               \|\sigma(u_1)-\sigma(u_2)\|^2_{\mathcal{L}^2_0}\leqslant K \|u_1-u_2\|^2_{H},
               \text{ and }
               \|\sigma(u)\|^2_{\mathcal{L}^2_0}\leqslant K(1+\|u\|^2_{H}),
           \end{align*}
           where $\mathcal{L}^2_0={\mathcal{L}^2}(U,H)$,
\item[({\bf H.4})]  for all $u_1,u_2,u\in D(A^{\alpha})$,
           \begin{align*}
           \|\sigma(u_1)-\sigma(u_2)\|^2_{\mathcal{L}^2_\alpha}\leqslant K \|u_1-u_2\|^2_{\alpha},
               \text{ and }
               \|\sigma(u)\|^2_{\mathcal{L}^2_{\alpha}}\leqslant K(1+\|u\|^2_{{\alpha}}),
           \end{align*}

\item[({\bf H.5})]  for all $u_1,u_2,u\in D(A^{\frac{1}{2}})$,
           \begin{align*}
           \|\sigma(u_1)-\sigma(u_2)\|^2_{\mathcal{L}^2_{\frac{1}{2}}}\leqslant K \|u_1-u_2\|^2_{\frac{1}{2}},
               \text{ and }
               \|\sigma(u)\|^2_{\mathcal{L}^2_{{\frac{1}{2}}}}\leqslant K(1+\|u\|^2_{{\frac{1}{2}}}).
           \end{align*}

\end{itemize}
 Set  $u^\varepsilon(t)=u^\varepsilon(t, \cdot)$, $n^\varepsilon(t)=n^\varepsilon(t, \cdot)$, and $c^\varepsilon(t)=c^\varepsilon(t, \cdot)$. Let $q>2$.

\begin{definition}\label{def mild solution}
We say that $(n^\varepsilon,c^\varepsilon,u^\varepsilon)$ is a mild solution of system (\ref{eq system 00}) if $(n^\varepsilon,c^\varepsilon,u^\varepsilon)$ is a progressively measurable stochastic processes with values in $C^0(\bar{\mathcal{O}})\times W^{1,q}({\mathcal{O}})\times D(A^\alpha)$, which satisfies,  $\mathbb{P}$-a.s.,
\begin{align}\label{eq system 01}
& n^\varepsilon(t)=e^{t\delta \Delta}n_0-\int_0^te^{(t-s)\delta \Delta}\Big\{u^\varepsilon(s)\cdot \nabla n^\varepsilon(s)\Big\}ds-\int_0^te^{(t-s)\delta \Delta}\Big\{\nabla\cdot\Big(\chi(c^\varepsilon(s))n^\varepsilon(s)\nabla c^\varepsilon(s)\Big)\Big\}ds,\nonumber\\
& c^\varepsilon(t)=e^{t\mu \Delta}c_0-\int_0^te^{(t-s)\mu \Delta}\Big\{u^\varepsilon(s)\cdot \nabla c^\varepsilon(s)\Big\}ds-\int_0^te^{(t-s)\mu \Delta}\Big\{k(c^\varepsilon(s))n^\varepsilon(s)\Big\}ds,\nonumber\\
& u^\varepsilon(t)=e^{-t\nu A}u_0-\int_0^te^{-(t-s)\nu A}\mathcal{P}\Big\{(u^\varepsilon(s)\cdot \nabla)u^\varepsilon(s)\Big\}ds-\int_0^te^{-(t-s)\nu A}\mathcal{P}\Big\{n^\varepsilon(s)\nabla\phi \Big\}ds \nonumber\\
&\quad\quad\quad\quad \quad\quad\quad\quad\quad\quad\quad\quad\quad\quad\quad\quad\quad+\int_0^te^{-(t-s)\nu A}\varepsilon\sigma(u^\varepsilon(s))dW_s.
\end{align}
\end{definition}

We point out that $(n^\varepsilon,c^\varepsilon,u^\varepsilon)$ is equivalent to a variational solution of the system in the Gelfand triple $W^{1,2}({\mathcal{O}})\subset L^2({\mathcal{O}})\subset W^{1,2}({\mathcal{O}})^*$, that is, $(n^\varepsilon,c^\varepsilon,u^\varepsilon)$  satisfies,
 $\mathbb{P}$-a.s.,
\begin{align}\label{eq system 02}
& n^\varepsilon(t)+\int_0^tu^\varepsilon(s)\cdot \nabla n^\varepsilon(s)ds=n_0+\delta \int_0^t\Delta n^\varepsilon(s)ds-\int_0^t\nabla\cdot\Big(\chi(c^\varepsilon(s))n^\varepsilon(s)\nabla c^\varepsilon(s)\Big)ds,\nonumber\\
& c^\varepsilon(t)+\int_0^tu^\varepsilon(s)\cdot \nabla c^\varepsilon(s)ds=c_0+\mu \int_0^t\Delta c^\varepsilon(s)ds-\int_0^tk(c^\varepsilon(s))n^\varepsilon(s)ds,\nonumber\\
& u^\varepsilon(t)+\int_0^t\mathcal{P}\Big\{(u^\varepsilon(s)\cdot \nabla)u^\varepsilon(s)\Big\}ds=u_0-\nu \int_0^tA u^\varepsilon(s)ds-\int_0^t\mathcal{P}\Big\{(n^\varepsilon(s)\nabla \phi)\Big\}ds\nonumber\\
&\quad\quad\quad\quad \quad\quad\quad\quad\quad\quad\quad\quad\quad\quad\quad\quad\quad+\int_0^t\varepsilon\sigma(u^\varepsilon(s))dW_s.
\end{align}

For more details of the above statement, we refer to Remark 2.1 in \cite{Zhai Zhang 2020}.

By Theorem 2.1 in \cite{Zhai Zhang 2020}, we have
\begin{proposition}\label{Thm main 1}

Assume
\begin{align}\label{eq boundary condition 3}
& n_0\in C^0(\bar{\mathcal{O}}),\ \ n_0>0\ \text{in}\ \bar{\mathcal{O}},\nonumber\\
& c_0\in W^{1,q}(\mathcal{O}),\ \text{for some }q>2,\ c_0>0\ \text{in}\ \bar{\mathcal{O}},\nonumber\\
& u_0\in D(A^\alpha),\ \text{for some}\ \alpha\in(1/2,1),
\end{align}
and the assumptions {\rm({\bf H.1})-({\bf H.5})} hold. Then there exists a unique mild/variational solution $(n^\varepsilon,c^\varepsilon,u^\varepsilon)$ to the system (\ref{eq system 00}).

Moreover, $\mathbb{P}$-a.s., for any $T>0$,
\begin{eqnarray}\label{ncu}
    (n^\varepsilon,c^\varepsilon,u^\varepsilon)\in
    L^\infty([0,T],C^0(\bar{\mathcal{O}}))
    \times
    L^\infty([0,T],W^{1,q}(\mathcal{O}))
    \times
L^\infty([0,T],D(A^\alpha)).
\end{eqnarray}
\end{proposition}

Prior to presenting the main result of this paper, we first prove the following result.
\begin{proposition}\label{regularity}
	Assume that requirements of Proposition \ref{Thm main 1} are met. Then, the solution $(n^\varepsilon,c^\varepsilon,u^\varepsilon)$ of (\ref{eq system 00}) enjoys the regularity properties:
	\begin{align*}
		(n^\varepsilon,c^\varepsilon,u^\varepsilon)\in
    C([0,T],C^0(\bar{\mathcal{O}}))
    \times
    C([0,T],C^0(\bar{\mathcal{O}}))
    \times
C([0,T],D(A^\alpha)).
	\end{align*}
\end{proposition}
\begin{proof}
 We will use the following facts:

 (F1)
  $\{e^{t \Delta}\}_{t>0}$ is strongly continuous in $C^0(\bar{\mathcal{O}})$ due to the maximum principle(see Theorem 8.1.1 in \cite{JJ}).

  (F2) $e^{-t A}$ is
     strong continuous  in $L^2(\mathcal{O})$.

     (F3) For $q>2$ and $\alpha\in(\frac{1}{2},1)$,  the continuous imbeddings $W^{1,q}(\mathcal{O})\hookrightarrow C^0(\bar{\mathcal{O}})$ and  $D(A^\alpha)\hookrightarrow C^0(\bar{\mathcal{O}})$ hold, respectively;

     (F4)
    Let $B$ denote the operator $-\Delta+1$ in $L^q(\mathcal{O})$ ($q>2$) equipped with the Neumann boundary condition.
    Then, for $\beta\in(\frac{1}{q},\frac{1}{2})$, we have the continuous imbedding $D(B^\beta)\hookrightarrow C^0(\bar{\mathcal{O}})$.

     For convenience of expression, we will denote $(n^\varepsilon,c^\varepsilon,u^\varepsilon)$ as $(n,c,u)$ in the following proof.
    By Definition \ref{def mild solution}, for any $0\leqslant s <t\leqslant T$,
    \begin{align}
    	& n(t)=e^{(t-s)\delta \Delta}n(s)-\int_s^te^{(t-l)\delta \Delta}\Big\{u(l)\cdot \nabla n(l)\Big\}dl-\int_s^te^{(t-l)\delta \Delta}\Big\{\nabla\cdot\Big(\chi(c(l))n(l)\nabla c(l)\Big)\Big\}dl,\nonumber\\
    	& c(t)=e^{(t-s)\mu \Delta}c(s)-\int_s^te^{(t-l)\mu \Delta}\Big\{u(l)\cdot \nabla c(l)\Big\}dl-\int_s^te^{(t-l)\mu \Delta}\Big\{k(c(l))n(l)\Big\}dl,\nonumber\\
    	& u(t)=e^{-(t-s)\nu A}u(s)-\int_s^te^{-(t-l)\nu A}\mathcal{P}\Big\{(u(l)\cdot \nabla)u(l)\Big\}dl-\int_s^te^{-(t-l)\nu A}\mathcal{P}\Big\{n(l)\nabla\phi \Big\}dl \nonumber\\
    	&\quad\quad\quad\quad \quad\quad\quad\quad\quad\quad\quad\quad\quad\quad\quad\quad\quad+\int_s^te^{-(t-l)\nu A}\sigma(u(l))dW_l.\nonumber
    \end{align}

    To simplify the exposition, we assume $\delta=\mu=\nu=1$, $\chi(c)=1$, and $k(c)=c$. The general case could be established by Condition (H.1)(a)(b) and $\|c\|_\infty\leqslant C\|c\|_{1,q}$.

    Using that $\nabla\cdot (nu)=u\cdot \nabla n$ due to the fact that $\nabla\cdot u=0 $, we can thus obtain,  for any $ t\in[s,T]$,
    \begin{align}\label{eq Lem2.1 01}
    	&\|n(t)-n(s)\|_{L^\infty}\nonumber\\
    	\leqslant& \|e^{(t-s) \Delta}n(s)-n(s)\|_{L^\infty}+ C \int_s^t\|B^\beta e^{-(t-l)(B-1)}\nabla \cdot (n\nabla c)(l)\|_{L^q}dl\nonumber\\
    	&+
    	C \int_s^t\|B^\beta e^{-(t-l)(B-1)}\Big(\nabla\cdot (un)(l)\Big)\|_{L^q}dl\nonumber\\
    \leqslant	&
    	\|e^{(t-s) \Delta}n(s)-n(s)\|_{L^\infty} + C \int_s^t (t-l)^{ -\frac{1}{2}-\beta}\| (n\nabla c)(l)\|_{L^q}dl
    	+
    	C \int_s^t (t-l)^{-\frac{1}{2} -\beta}\|(un)(l)\|_{L^q}dl\nonumber\\
    	\leqslant&
    	\|e^{(t-s) \Delta}n(s)-n(s)\|_{L^\infty} + C (t-s)^{\frac{1}{2}-\beta }\sup_{s\leqslant l\leqslant t}\big(\|n(l)\|_{L^\infty}\cdot\|c(l)\|_{1,q}+\|u(l)\|_\alpha\cdot\|n(l)\|_{L^\infty}\big).
    \end{align}
    We have used the following property for semigroup $e^{-tB}$; see  Lemma $1.3(\romannumeral4)$ in \cite{Winkler 10}:
 \begin{align}
        \|B^\beta e^{-t(B-1)}\nabla\cdot w\|_{L^q}\leqslant C(1+ t^{-\frac{1}{2}-\beta })\|w\|_{L^q}, \ \ {\rm for}\  w\in L^q(\mathcal{O}).
    \end{align}


    By combining equations (\ref{eq Lem2.1 01}) and (\ref{ncu}), we can observe that,
    \begin{equation*}
    	\lim\limits_{t\rightarrow s}\|n(t)-n(s)\|_{L^\infty}=0,
    \end{equation*}
    which implies that $n\in C([0,T],C^0(\bar{\mathcal{O}}))$.

    Proceeding similarly as (\ref{eq Lem2.1 01}), for $\gamma \in (\frac{1}{2},1)$ and $ t\in[s,T]$,
    \begin{align}\label{eq Lem2.1 03}
    	&\|c(t)-c(s)\|_{L^\infty}\nonumber\\
    	\leqslant& \|e^{(t-s) \Delta}c(s)-c(s)\|_{L^\infty}+ \int_s^t\|B^\gamma e^{-(t-l)(B-1)}\Big\{u(l)\cdot \nabla c(l)+c(l)n(l)\Big\}\|_{L^q}dl\nonumber\\
    	\leqslant&
    	\|e^{(t-s) \Delta}c(s)-c(s)\|_{L^\infty} + \int_s^t (t-l)^{-\gamma}\|u(l)\cdot \nabla c(l)+c(l)n(l)\|_{L^q}dl\nonumber\\
    	\leqslant&
    	\|e^{(t-s) \Delta}c(s)-c(s)\|_{L^\infty} + (t-s)^{1-\gamma}\sup_{s\leqslant l\leqslant t}\big(\|u(l)\|_\alpha\cdot\|c(l)\|_{1,q}+\|c(l)\|_{1,q}\cdot\|n(l)\|_{L^\infty}\big).
    \end{align}
    Thus, by combining equations (\ref{ncu}) and (\ref{eq Lem2.1 03}), we conclude that
    \begin{equation*}
    	\lim\limits_{t\rightarrow s}\|c(t)-c(s)\|_{L^\infty}=0,
    \end{equation*}
    this implies that $c\in C([0,T],C^0(\bar{\mathcal{O}}))$.

    For $u(t)$, the following inequality holds:
    \begin{align}\label{eq Lem2.1 04}
    	& \|A^\alpha u(t)-A^\alpha u(s)\|_{L^2}\nonumber\\
    	\leqslant&
    	\|e^{-(t-s)A}A^\alpha u(s)-A^\alpha u(s)\|_{L^2} + \int_s^t\|e^{-(t-l)A}A^\alpha \mathcal{P}\{(u(l)\cdot \nabla)u(l)\}\|_{L^2}dl\nonumber\\
    	&+
    	\int_s^t\|e^{-(t-l)A}A^\alpha \mathcal{P}\{n(l)\nabla\phi\}\|_{L^2} dl
    	+
    	\|\int_s^te^{-(t-l)A}A^\alpha\sigma(u(l))dW_l\|_{L^2}\nonumber\\
    	:=&
    	\|e^{-(t-s)A}A^\alpha u(s)-A^\alpha u(s)\|_{L^2}+I_1(t,s)+I_2(t,s)+I_3(t,s).
    \end{align}
    Noticing the inequality:
    \begin{eqnarray}\label{eq 2024 02 03}
        \|(u\cdot\nabla )u\|_{L^2}\leqslant\|u\|_{L^\infty}\|\nabla u\|_{L^2} \leqslant C \|A^\alpha u\|^2_2.
    \end{eqnarray}
    We can conclude that:
    \begin{align}\label{eq Lem2.1 I1}
    	I_1(t,s)
    	\leqslant
    	C\int_s^t(t-l)^{-\alpha}\|A^\alpha u(l)\|^2_2dl
    	\leqslant
    	C(t-s)^{1-\alpha}\sup_{s\leqslant l\leqslant t}\| u(l)\|^2_\alpha.
    \end{align}

    Regarding $I_2$, we have:
    \begin{align}\label{eq Lem2.1 I2}
    	I_2(t,s)
    	&\leqslant
    	\int_s^t(t-l)^{-\alpha}\|n(l)\nabla\phi\|_{L^2}dl
    	\leqslant
    	\sup_{s\leqslant l\leqslant t}\|n(l)\|_{L^\infty} \|\nabla\phi\|_{L^\infty}(t-s)^{1-\alpha}\nonumber\\
    	&\leqslant
    	C\sup_{s\leqslant l\leqslant t}\|n(l)\|_{L^\infty} (t-s)^{1-\alpha}.
    \end{align}

    To estimate $I_3$, according to assumption (H.4) and (\ref{ncu}), we can easily obtain
    \begin{align*}
    	\int_s^T
    	\|A^\alpha\sigma(u(l))\|_{\mathcal{L}^2_0}^2dl
    	\leqslant C\int_s^T \|A^\alpha u(l)\|_{L^2}^2dl<\infty,\ \mathbb{P}\text{-a.s.}.
    \end{align*}
 Then it is well-known that $\int_s^te^{-(t-l)A}A^\alpha\sigma(u(l))dW_l$ has a continuous modification in $H$; see, e.g., { \cite[Theorem 6.10(page 160)]{Da Prato}}.
    Hence \begin{align}\label{eq Lem2.1 I3}
    	\lim_{t\rightarrow s}I_3(t)=I_3(s)=0
    \end{align}

    By combining (\ref{eq Lem2.1 04})--(\ref{eq Lem2.1 I3}),
    \begin{align}
    	\lim\limits_{t\rightarrow s}\|A^\alpha u(t)-A^\alpha u(s)\|_{L^2}
    	=0,\nonumber
    \end{align}
    implying that $u\in C([0,T],D(A^\alpha))$. The proof of Lemma 3.1 is complete.\hfill $\Box$
\end{proof}	
\vskip 0.2cm

The purpose of this paper is to establish a large deviation principle for the solution of (\ref{eq system 00}), denoted as $(n^\varepsilon,c^\varepsilon,u^\varepsilon)$ on $C([0,T],C^0(\bar{\mathcal{O}}))\times C([0,T],C^0(\bar{\mathcal{O}}))\times
C([0,T],D(A^\alpha)$ as $\varepsilon\rightarrow 0$. Before presenting the main result, we introduce the definition of LDP. Let $(\mathcal{E},\rho)$ be a Polish space. A lower-semicontinuous  function $I:\mathcal{E}\rightarrow [0,\infty]$ is called a rate function if the level set $\{\varphi\in \mathcal{E}: I(\varphi)\leqslant M\}$ is a compact subset of $\mathcal{E}$ for all $M\geqslant 0$.
\begin{definition}
	A family of $\mathcal{E}$-valued random variables $\{X^\varepsilon\}_{\varepsilon>0}$ is said to satisfy the LDP on $\mathcal{E}$ with rate function $I$ if for each Borel subset $B$ of $\mathcal{E}$,
	\begin{align*}
		-\inf_{\varphi\in\mathring{B}}I(\varphi)\leqslant
		\liminf_{\varepsilon\rightarrow 0}\varepsilon^2\log\mathbb{P}(X^\varepsilon\in B)
		\leqslant
		\limsup_{\varepsilon\rightarrow0}\varepsilon^2\log\mathbb{P}(X^\varepsilon\in B)
		\leqslant
		-\inf_{\varphi\in\overline{B}}I(\varphi),
	\end{align*}
 {where $\mathring{B}$ and $\overline{B}$ are the interior and the closure of $B$ in $\mathcal{E}$}, respectively.
\end{definition}

We also need to introduce the so-called skeleton equation, which is used to define the rate function in our main result, for any $h\in L^2([0,T],U)$,
\begin{align}\label{eq skeleton}
& dn^h+u^h\cdot \nabla n^h dt=\delta \Delta n^h dt-\nabla\cdot(\chi(c^h)n^h\nabla c^h)dt,\nonumber\\
& dc^h+u^h\cdot \nabla c^h dt=\mu \Delta c^h dt-k(c^h)n^h dt,\nonumber\\
& du^h+(u^h\cdot \nabla)u^h dt+\nabla P^h dt=\nu \Delta u^h dt-n^h\nabla\phi ~dt+\sigma(u^h )h dt,\\
& \nabla\cdot u^h=0,\ \ \ \ \ t>0,\ x\in\mathcal{O},\nonumber
\end{align}
with the boundary conditions
\begin{align}\label{eq h boundary condition 1}
\frac{\partial n^h}{\partial v}=\frac{\partial c^h}{\partial v}=0\text{  and  }u^h=0\ \ \text{   for }x\in\partial\mathcal{O}\text{ and }t>0,
\end{align}
and the initial conditions
\begin{align}\label{eq h boundary condition 2}
n^h(0,x)=n_0(x),\ \ \ c^h(0,x)=c_0(x),\ \ \ u^h(0,x)=u_0(x),\ \ \ x\in\mathcal{O}.
\end{align}
Here $L^2([0, T] , U)$ denotes the set of square integrable $U$-valued functions on $[0, T]$ endowed with the norm
$$\|\varphi\|_{L^2([0, T] , U)}^2=\int_{0}^{T}\|\varphi(t)\|^2_Udt.$$
We state the following result, whose proof is provided in Section \ref{skeleton}.
\begin{proposition}\label{Thm skeleton1}
	Assume that requirements of Proposition \ref{Thm main 1} are met. Then, for any $h\in L^2([0,T],U)$, there exists a unique mild/variational solution $(n^h,c^h,u^h)\in
 C([0,T],C^0(\bar{\mathcal{O}}))
    \times
    C([0,T],C^0(\bar{\mathcal{O}}))
    \times
C([0,T],D(A^\alpha))$ to the system (\ref{eq skeleton}) with the boundary-initial condition (\ref{eq h boundary condition 1}) and (\ref{eq h boundary condition 2}). That is, $(n^h,c^h,u^h)$ satisfies:
	\begin{align}\label{eq skeleton 1}
		& n^h(t)=e^{t\delta \Delta}n_0-\int_0^te^{(t-s)\delta \Delta}\Big\{u^h(s)\cdot \nabla n^h(s)\Big\}ds-\int_0^te^{(t-s)\delta \Delta}\Big\{\nabla\cdot\Big(\chi(c^h(s))n^h(s)\nabla c^h(s)\Big)\Big\}ds,\nonumber\\
		& c^h(t)=e^{t\mu \Delta}c_0-\int_0^te^{(t-s)\mu \Delta}\Big\{u^h(s)\cdot \nabla c^h(s)\Big\}ds-\int_0^te^{(t-s)\mu \Delta}\Big\{k(c^h(s))n^h(s)\Big\}ds,\nonumber\\
		& u^h(t)=e^{-t\nu A}u_0-\int_0^te^{-(t-s)\nu A}\mathcal{P}\Big\{(u^h(s)\cdot \nabla)u^h(s)\Big\}ds-\int_0^te^{-(t-s)\nu A}\mathcal{P}\Big\{n^h(s)\nabla\phi \Big\}ds \nonumber\\
		&\quad\quad\quad\quad \quad\quad\quad\quad\quad\quad\quad\quad\quad\quad\quad\quad\quad+\int_0^te^{-(t-s)\nu A}\sigma(u^h(s))h(s)ds.
	\end{align}
\end{proposition}

We now formulate the main result.
\begin{theorem}\label{thm}
    Assume that requirements of Proposition \ref{Thm main 1} are met.
Then as $\varepsilon\rightarrow 0$, the  solution to the system (\ref{eq system 00}) $\{(n^\varepsilon,c^\varepsilon,u^\varepsilon)\}_{\varepsilon>0}$ satisfies  LDP on $C([0,T],C^0(\bar{\mathcal{O}}))\times C([0,T],C^0(\bar{\mathcal{O}}))\times
C([0,T],D(A^\alpha)$ with the rate function $I$, given by
    	$$I(\varphi)=\inf_{h\in L^2([0, T] ,U)}\left\lbrace\frac{1}{2}\int_{0}^T\|h(s)\|_U^2ds:\varphi=(n^h,c^h,u^h)\right\rbrace, $$
     where $(n^h,c^h,u^h)$ is the solution to the skeleton equation (\ref{eq skeleton}). Here, we use the convention that the infimum of an empty set is $\infty$.

\end{theorem}

\begin{proof}
Proposition \ref{Thm skeleton1} implies that there exists a measurable mapping $$\mathcal{G}^0: L^2([0, T] , U)\rightarrow C([0,T],C^0(\bar{\mathcal{O}}))\times C([0,T],C^0(\bar{\mathcal{O}}))\times
C([0,T],D(A^\alpha),$$
such that $\mathcal{G}^0(h)$ is the solution to (\ref{eq skeleton}) for any $h\in L^2([0, T] , U)$.

By Yamada-Watanabe's Theorem (cf. \cite{RSZ}) and Propositions \ref{Thm main 1}, \ref{regularity}, there exists a measurable map
$$\mathcal{G}^\varepsilon: C([0,T],U)\rightarrow C([0,T],C^0(\bar{\mathcal{O}}))\times C([0,T],C^0(\bar{\mathcal{O}}))\times
C([0,T],D(A^\alpha),$$
such that $\mathcal{G}^\varepsilon(W_\cdot)$ is the solution to (\ref{eq system 00}) for any $U$-valued Brownian motion $W_\cdot$.

For any $N\in \mathbb{N}$, set
\begin{equation*}
	S^N:=\left\lbrace \varphi\in L^2([0,T],U):\|\varphi\|_{L^2([0, T] ,U)}^2\leqslant N\right\rbrace .
\end{equation*}
$S^N$ endowed with the weak topology of the space $L^2([0, T],U)$ is a Polish space(see e.g.\cite{BDM}). We also define
$$\mathcal{P}^N:=\left\lbrace \Phi :{ \Phi\  is\ U\text{-}valued\ \mathcal{F}_t\text{-}predictable\ process\ such\ that\ \Phi(\omega)\in S^N,\mathbb{P}\text{-}a.s.}\right\rbrace. $$

Using \cite[Theorem 3.1]{MSZ}, in order to prove Theorem \ref{thm}, we only need to verify the following condition:
\begin{Cod}\label{cod}
(a)For every $N< \infty$, $\{h^i\}_{i\in\mathbb{N}}\subset S^N$ that converges weakly to some element $h$ in $L^2([0,T],U)$ as $i\rightarrow \infty$,
	$$\lim\limits_{i\rightarrow \infty}\rho (\mathcal{G}^0\left( h^i\right),\mathcal{G}^0\left( h\right))=0.$$
	
	(b)For every $N<\infty$, any family $\{h^\varepsilon\}_{\varepsilon>0}\subset \mathcal{P}_N$ and any $\delta>0$
	$$\lim\limits_{\varepsilon\rightarrow 0}\mathbb{P}(\rho (Y^\varepsilon,Z^\varepsilon)>\delta)=0.$$
	Here $Y^\varepsilon=\mathcal{G}^\varepsilon(W_\cdot+\frac{1}{\varepsilon}\int_0^\cdot h^\varepsilon(s)ds),Z^\varepsilon=\mathcal{G}^0( h^\varepsilon)$ and $\rho$ is the metric of  $C([0,T],C^0(\bar{\mathcal{O}}))\times C([0,T],C^0(\bar{\mathcal{O}}))\times
C([0,T],D(A^\alpha)$.
\end{Cod}
We will check (a) and (b) in Section \ref{a}  and \ref{b}, respectively.
\hfill $\Box$
\end{proof}

We end this section by introducing the following generalized Gronwall-Bellman inequality, which plays an important role in Sections \ref{a}  and \ref{b}. The proof is similar to that of  \cite{W} and \cite{YGD}, and we omit it here.

\begin{lemma} \label{GGI}

    Let $0 <\alpha_1\leqslant\alpha_2\leqslant\alpha_3< 1$ and consider the time interval $I = [0, T)$, where $T \leqslant \infty$. Suppose $a(t)$ is a
nonnegative function that is locally integrable on $I$, and $b(t), g_1(t),g_2(t),$ and $g_3(t)$ are nonnegative, nondecreasing
continuous functions defined on $I$ and bounded by a positive constant $M$. If $f(t)$ is nonnegative
and locally integrable on $I$, and satisfies
\begin{align}\label{ft}
    f(t)\leqslant a(t)+b(t)\int_0^tf(s)ds+\sum_{\xi=1}^3g_\xi(t)\int_0^t(t-s)^{\alpha_\xi-1}f(s)ds,
\end{align}
then
\begin{align}\label{ft}
    f(t)\leqslant a(t)+&\sum_{n=1}^\infty\sum_{\substack{i,j,k,l\geqslant 0\\i+j+k+l=n}}\frac{n!(b(t))^i(g_1(t)\Gamma(\alpha_1))^j(g_2(t)\Gamma(\alpha_2))^k(g_3(t)\Gamma(\alpha_3))^l}{i!j!k!l!\Gamma(i+j\alpha_1+k\alpha_2+l\alpha_3)}\nonumber\\
    &\times\int_0^t(t-s)^{i+j\alpha_1+k\alpha_2+l\alpha_3-1}a(s)ds,
\end{align}
where $\Gamma(\cdot)$ is the Gamma function. Moreover, if $a(t)$ is nondecreasing on $I$. Then
\begin{align}
    f(t)\leqslant a(t)E_{\alpha_1}(g(t)\Gamma(\alpha_1)t^{\alpha_1})E_{\alpha_2}(g(t)\Gamma(\alpha_2)t^{\alpha_2})E_{\alpha_3}(g(t)\Gamma(\alpha_3)t^{\alpha_3})e^{\frac{b(t)t}{\alpha_1}}.
\end{align}
Here, the Mittag-Leffler function  $E_{\varrho}(z)$ is  defined by $E_{\varrho}(z)=\sum\limits_{k=0}^\infty \frac{z^k}{\Gamma (k\varrho+1)}$ for $z > 0$, and $E_{\varrho}(z)<\infty$ when  $\varrho>0$.
\end{lemma}

\section{Proof of Proposition \ref{Thm skeleton1}}\label{skeleton}
\setcounter{equation}{0}
 In this section, our aim is to prove the following proposition regarding the existence, uniqueness of solutions for the skeleton equations (\ref{eq skeleton}) and provide a priori estimates for them. This proposition implies Proposition \ref{Thm skeleton1}.
The a priori estimates are devoted to obtain (a) in the proof of Theorem \ref{Thm main 1}.
\begin{proposition}\label{Thm skeleton}
	Assume that requirements of Proposition \ref{Thm main 1} are met. Then, for any $h\in S^N$, there exists a unique mild/variational solution $(n^h,c^h,u^h)\in
 C([0,T],C^0(\bar{\mathcal{O}}))
    \times
    C([0,T],C^0(\bar{\mathcal{O}}))
    \times
C([0,T],D(A^\alpha))$ to the system (\ref{eq skeleton}) with the boundary-initial condition (\ref{eq h boundary condition 1}) and (\ref{eq h boundary condition 2}), and there exists a constant $C$ depending on ${N,T,\|n_0\|_{L^\infty},\|c_0\|_{1,q},\|u_0\|_\alpha}$ such that
\begin{align}\label{CNT}
    \sup_{h\in S^N}\Big(
    \sup_{0\leqslant t\leqslant T}\|n^h(t)\|_{L^\infty}^2+\sup_{0\leqslant t\leqslant T}\|c^h(t)\|_{1,q}^2+\sup_{0\leqslant t\leqslant T}\|A^{\alpha}u^h(t)\|_{L^2}^2
    \Big)
    \leqslant C.
\end{align}
\end{proposition}

The proof of proposition \ref{Thm skeleton} is a generalization of \cite[Theorem 2.1]{Zhai Zhang 2020} and Proposition \ref{regularity}. The main difference lies in the absence of a diffusion term, but the presence of an additional drift term given by $\int_0^te^{-(t-s)\nu A}\sigma(u^h(s))h(s)ds$, requiring appropriate estimates. To simplify the exposition, in this section, we {abbreviate $(n^h,c^h,u^h)$ as $(n,c,u)$} and assume $\delta=\mu=\nu=1$, $\chi(c)=1$, and $k(c)=c$. 

\subsection{Existence of Local Solutions}\label{subsection1}

Introduce the following spaces
$$
\Upsilon^n_t:=L^\infty([0,t],C^0(\bar{\mathcal{O}})),\ \Upsilon^c_t:=L^\infty([0,t],W^{1,q}(\mathcal{O})),
\ \Upsilon^u_t:=L^\infty([0,t],D(A^\alpha))
$$
with the corresponding norms given by
$$
\|n\|_{\Upsilon^n_t}=\sup_{s\in[0,t]}\|n(s)\|_{L^\infty},\ \|c\|_{\Upsilon^c_t}=\sup_{s\in[0,t]}\|c(s)\|_{1,q},\ \|u\|_{\Upsilon^u_t}=\sup_{s\in[0,t]}\|u(s)\|_{\alpha}.
$$

\begin{proposition}\label{thm local}
There exist $\kappa_{max}\in (0,\infty]$ and a unique local solution $(n,c,u)\in
 C([0,\kappa_{max}),C^0(\bar{\mathcal{O}}))
    \times
    C([0,\kappa_{max}),C^0(\bar{\mathcal{O}}))
    \times
C([0,\kappa_{max}),D(A^\alpha))$ of system (\ref{eq skeleton}) on the interval $[0,\kappa_{max})$ such that, if $\kappa_{max}<\infty$, then
	\begin{align}\label{eq Tmax}
		 \|n\|_{\Upsilon^n_t}+\|c\|_{\Upsilon^c_t}+\|u\|_{\Upsilon^u_t}\rightarrow \infty\ \text{as } t\uparrow\kappa_{max}.
	\end{align}
\end{proposition}
\begin{proof}
We will modify the coefficients in system (\ref{eq skeleton}) in order to apply a cut-off argument.
Fix a function $\theta\in C^2([0,\infty),[0,1])$ such that
\begin{itemize}
\item[(1)] $\theta(r)=1$, $r\in[0,1]$,

\item[(2)] $\theta(r)=0$, $r>2$,

\item[(3)] $\sup_{r\in[0,\infty)}|\theta'(r)|\leqslant C<\infty$.
\end{itemize}
Set $\theta_m(\cdot)=\theta(\frac{\cdot}{m})$. For every $m\geqslant 1$, consider the following system of PDEs:
\begin{align}\label{eq truc n c}
& dn+\theta_m(\|u\|_{\Upsilon^u_t})\theta_m(\|n\|_{\Upsilon^n_t})u\cdot \nabla ndt
  =
  \delta \Delta ndt-\theta_m(\|n\|_{\Upsilon^n_t})\theta_m(\|c\|_{\Upsilon^c_t})\nabla\cdot(\chi(c)n\nabla c)dt,\nonumber\\
& dc+\theta_m(\|u\|_{\Upsilon^u_t})\theta_m(\|c\|_{\Upsilon^c_t})u\cdot \nabla cdt
  =
  \mu \Delta c dt-\theta_m(\|c\|_{\Upsilon^c_t})\theta_m(\|n\|_{\Upsilon^n_t})k(c)ndt,\nonumber\\
& du+\theta_m(\|u\|_{\Upsilon^u_t})(u\cdot \nabla)udt+\nabla Pdt
   =
   \nu \Delta udt-\theta_m(\|n\|_{\Upsilon^n_t})n\nabla\phi dt+\sigma(u)hdt,\nonumber\\
& \nabla\cdot u=0,\ \ \ \ \ t>0,\ x\in\mathcal{O}.
\end{align}


We consider the Banach space:
$$S_T=\{(n,c,u)\in L^\infty([0,T],C^0(\bar{\mathcal{O}})\times W^{1,q}(\mathcal{O})\times D(A^\alpha)):n(0)=n_0,c(0)=c_0,u(0)=u_0\}
$$
with the corresponding norms given by
$$
\|(n,c,u)\|^2_{S_T}:=\|n\|^2_{\Upsilon^n_T}+\|c\|^2_{\Upsilon^c_T}
+
\|u\|^2_{\Upsilon^u_T}.
$$


We introduce a mapping $\Phi=(\Phi_1,\Phi_2,\Phi_3)$ on $S_T$ by defining
\begin{align}
  \Phi_1(n,c,u)(t)
:=&
  e^{t\Delta}n_0
  -
  \int_0^te^{(t-s)\Delta}\Big\{\theta_m(\|n\|_{\Upsilon^n_s})\theta_m(\|c\|_{\Upsilon^c_s})\nabla\cdot(n\nabla c)\nonumber\\
  &\ \ \ \ \ \ \ \ \ \ \ \ \ \ \ \ \ \ \ \ \ \ \ \ \ \ \ +
  \theta_m(\|u\|_{\Upsilon^u_s})\theta_m(\|n\|_{\Upsilon^n_s})\nabla\cdot (un)\Big\}(s)ds,\nonumber
\end{align}
\begin{align}
  \Phi_2(n,c,u)(t)
:=&
  e^{t\Delta}c_0
  -
  \int_0^te^{(t-s)\Delta}\Big\{\theta_m(\|n\|_{\Upsilon^n_s})\theta_m(\|c\|_{\Upsilon^c_s})nc\nonumber\\
      &\ \ \ \ \ \ \ \ \ \ \ \ \ \ \ \ \ \ \ \ \ \ +
      \theta_m(\|u\|_{\Upsilon^u_s})\theta_m(\|c\|_{\Upsilon^c_s})u\cdot\nabla c\Big\}(s)ds,\nonumber
\end{align}
and
\begin{align}
\Phi_3(n,c,u)(t):=&e^{-tA}u_0-\int_0^te^{-(t-s)A}\theta_m(\|u\|_{\Upsilon^u_s})\mathcal{P}\{(u(s)\cdot \nabla)u(s)\}ds\nonumber\\
                   &
                   -\int_0^te^{-(t-s)A}\theta_m(\|n\|_{\Upsilon^n_s})\mathcal{P}\{n(s)\nabla\phi\} ds\nonumber\\
                   &
                   +
                  \int_0^te^{-(t-s)A}\mathcal{P}\{\sigma(u(s))h(s)\}ds.\nonumber
\end{align}

Recall the following result from (3.5) and (3.6) in \cite{Zhai Zhang 2020}, for $\beta\in(\frac{1}{q},\frac{1}{2})$ and $\gamma \in (\frac{1}{2},1)$,
\begin{align}\label{eq Phi1 esta1}
  &\|\Phi_1(n,c,u)\|_{\Upsilon^n_T}\leqslant
  \|n_0\|_{L^\infty} + C m^2T^{\frac{1}{2}-\beta},
\end{align}
\begin{align}\label{eq Phi2 esta1}
  &\|\Phi_2(n,c,u)\|_{\Upsilon^c_T}\leqslant
  \|c_0\|_{1,q} + C m^2 T^{1-\gamma}.
\end{align}

For $\Phi_3$, by using assumption (H.4) and the H\"older inequality,  and applying similar arguments as in proving (3.7), (3.8) and (3.9) in \cite{Zhai Zhang 2020}, we can therefore yield, for any $t\in [0,T]$,
\begin{align}\label{eq Phi3 esta1}
 & \|A^\alpha\Phi_3(n,c,u)(t)\|_{L^2}\nonumber\\
\leqslant&
   \|e^{-tA}A^\alpha u_0\|_{L^2} + \int_0^t\|e^{-(t-s)A}A^\alpha \theta_m(\|u\|_{\Upsilon^u_s})\mathcal{P}\{(u(s)\cdot \nabla)u(s)\}\|_{L^2}ds\nonumber\\
   &+
   \int_0^t\|e^{-(t-s)A}A^\alpha \theta_m(\|n\|_{\Upsilon^n_s})\mathcal{P}\{n(s)\nabla\phi\}\|_{L^2} ds+
   \int_0^t\|e^{-(t-s)A}A^\alpha\mathcal{P}\{\sigma(u(s))h(s)\}\|_{L^2}ds\nonumber\\
\leqslant&
\|u_0\|_{\alpha}+C\int_0^t(t-s)^{-\alpha}\theta_m(\|u\|_{\Upsilon^u_s})\|A^\alpha u(s)\|^2_{L^2}ds\nonumber\\
   &+\int_0^t(t-s)^{-\alpha}\theta_m(\|n\|_{\Upsilon^n_s})\|n(s)\nabla\phi\|_{L^2}ds+\int_0^t\|A^\alpha\sigma(u(s))\|_{\mathcal{L}^2_0}\|h(s)\|_{U}ds\nonumber\\
   \leqslant&\|u_0\|_{\alpha}+ Cm^2t^{1-\alpha}+Cmt^{1-\alpha}+C_Nt(1+m).
\end{align}
{We have used (\ref{eq 2024 02 03}) and the following property for semigroup $e^{-tA}$:
    \begin{align}\label{Aalpha}
        \| e^{-tA}A^\alpha w\|_{L^2}\leqslant t^{-\alpha}\|w\|_{L^2}, \ \ {\rm for}\ w\in L^2.
    \end{align}
    }
Now, (\ref{eq Phi1 esta1}), (\ref{eq Phi2 esta1}) and (\ref{eq Phi3 esta1}) together show that $\Phi$ maps $S_T$ into itself.

Next we will prove that if $T>0$ is small enough, then $\Phi$ can be made  a contraction on $S_T$.

Let $(n_1,c_1,u_1), (n_2,c_2,u_2)\in S_T$. Also similiar as the discussion of proving (3.17) and (3.18) in \cite{Zhai Zhang 2020}, for $\beta\in (\frac{1}{q},\frac{1}{2})$ and $\gamma\in (\frac{1}{2}, 1)$, we have,
\begin{align}\label{eq Phi1 1-2}
	&\|\Phi_1(n_1,c_1,u_1)-\Phi_1(n_2,c_2,u_2)\|_{\Upsilon^n_T}\nonumber\\
	\leqslant&
	C m\Big(\|c_1-c_2\|_{\Upsilon^c_T}+\|u_1-u_2\|_{\Upsilon^u_T}+\|n_1-n_2\|_{\Upsilon^n_T}\Big)T^{\frac{1}{2}-\beta},
\end{align}
and
\begin{align}\label{eq Phi2 1-2}
	&  \|\Phi_2(n_1,c_1,u_1)-\Phi_2(n_2,c_2,u_2)\|_{\Upsilon^c_T}\nonumber\\
	\leqslant&
	C m\Big(\|c_1-c_2\|_{\Upsilon^c_T}+\|u_1-u_2\|_{\Upsilon^u_T}+\|n_1-n_2\|_{\Upsilon^n_T}\Big)T^{1-\gamma}.
\end{align}

In order to estimate $\|\Phi_3(n_1,c_1,u_1)-\Phi_3(n_2,c_2,u_2)\|_{\Upsilon^u_T}$, we use (H.4), {(\ref{Aalpha}) } and similarly discuss as proving (3.22) and (3.20) in \cite{Zhai Zhang 2020}, for any $t\in [0,T]$,
\begin{align}\label{eq u1-u2}
&    \|\Phi_3(n_1,c_1,u_1)(t)-\Phi_3(n_2,c_2,u_2)(t)\|_\alpha\nonumber\\
\leqslant&
    C \int_0^t(t-s)^{-\alpha}\|\theta_m(\|u_1\|_{\Upsilon^u_s})(u_1\cdot \nabla)u_1-\theta_m(\|u_2\|_{\Upsilon^u_s})(u_2\cdot \nabla)u_2\|_{L^2}ds\nonumber\\
    &+
    C \int_0^t(t-s)^{-\alpha}\|\theta_m(\|n_1\|_{\Upsilon^n_s})n_1\nabla\phi-\theta_m(\|n_2\|_{\Upsilon^n_s})n_2\nabla\phi\|_{L^2}ds\nonumber\\
    &+
    C \int_0^t\|(A^\alpha\sigma(u_1)-A^\alpha\sigma(u_2))h(s)\|_{L^2}ds\nonumber\\
\leqslant &
   C m \|u_1-u_2\|_{\Upsilon^u_t} t^{1-\alpha}+C \|n_1-n_2\|_{\Upsilon^n_T} \cdot t^{1-\alpha}+C \int_0^t\|\sigma(u_1)-\sigma(u_2)\|_{\mathcal{L}^2_\alpha}\|h(s)\|_Udt\nonumber\\
	\leqslant&
	C_N (t+mt^{1-\alpha}) \|u_1-u_2\|_{\Upsilon^u_t} +C \|n_1-n_2\|_{\Upsilon^n_T} \cdot t^{1-\alpha}.
\end{align}

 By virtue of (\ref{eq Phi1 1-2}) (\ref{eq Phi2 1-2}) and (\ref{eq u1-u2}), one can find constants $\rho,\ C_{N,m,T}>0$ such that
\begin{align}\label{eq fix point}
    \|\Phi(n_1,c_1,u_1)-\Phi(n_2,c_2,u_2)\|^2_{S_T}
\leqslant
    C_{N,m,T}T^\rho \|(n_1,c_1,u_1)-(n_2,c_2,u_2)\|^2_{S_T},
\end{align}
and
$$
\lim_{T\rightarrow0}C_{N,m,T}T^\rho=0.
$$

Finally using classical arguments and similar arguments as proving Proposition \ref{regularity}, we can obtain Proposition \ref{thm local}.



\hfill $\Box$
\end{proof}

\subsection{ Global existence}\label{subsection2}
According to Proposition \ref{thm local}, we possess a unique
local-in-time solution for (\ref{eq skeleton}) that has been extended up to a maximal time $\kappa_{max}\leqslant \infty$. In order to prove its global existence, we just need to establish the following estimate
\begin{eqnarray}\label{k_max}
\|n\|_{\Upsilon^n_{T\wedge \kappa_{max}}}\vee\|c\|_{\Upsilon^c_{T\wedge \kappa_{max}}}\vee\|u\|_{\Upsilon^u_{T\wedge \kappa_{max}}}\leqslant C_{N,T,\|n_0\|_{L^\infty},\|c_0\|_{1,q},\|u_0\|_\alpha}, \ \ \ {\rm for\ any}\ T\in (0,\infty),
\end{eqnarray}
which allows for an application of (\ref{eq Tmax}) to rule out the case $\kappa_{max}<\infty$. It should be noted that (\ref{CNT}) holds once  (\ref{k_max}) has been proven. In this subsection, the constant $C,C_T,\cdots$ may depend on $\|n_0\|_{L^\infty},\|c_0\|_{1,q},\|u_0\|_\alpha$, and we will omit it for simplicity.
To get (\ref{k_max}), we shall first recall the following results from Corollary 4.2 and Lemma 4.5 in \cite{Winkler}.
\begin{lemma}\label{lem 3.1}
Let $p>1$, $r\in [1, \frac{p}{p-1}]$ and $0<T<\kappa_{max}$. Then there exists a constant $C_T$ and $C_p$ such that
\begin{equation}\label{3.1.0}
\int_0^{T}\|n(t)\|_{L^p}^rdt
\leqslant
 C_T\left (\int_0^{T}\int_{\mathcal{O}} \frac{|\nabla n(t,x)|^2}{n(t,x)}dxdt +1\right )^{\frac{(p-1)r}{p}},
\end{equation}
and
\begin{equation}\label{3.1.1}
\int_{\mathcal{O}}n^p(t,x)dx
\leqslant (\int_{\mathcal{O}}n^p(0,x)dx+1)e^{C_p\int_0^t\int_{\mathcal{O}}|\nabla c(s,x)|^4dxds},\ t\in[0,{T}].
\end{equation}
\end{lemma}
We start with an estimate of the $L^2$ norm of $u$ and $\nabla u$.
\begin{lemma}\label{lem 3.2}
Let $\theta\in (0,1)$, $h\in S^N$ and $0<T<\kappa_{max}$. Then there exists a constant $C_{N,T}$ such that
\begin{align}\label{3.1}
&\sup_{0\leqslant t\leqslant {T}}\|u(t)\|_{L^2}^2+\int_0^{T}\|\nabla u(t)\|_{L^2}^2dt
\leqslant
C_{N,T}\Big (\int_0^{T}\int_{\mathcal{O}} \frac{|\nabla n(t,x)|^2}{n(t,x)}dxdt +1\Big )^{\frac{\theta}{2}}.
\end{align}
Moreover,
\begin{align}\label{3.2}
\int_0^{T}\int_{\mathcal{O}}|u(t,x)|^4dxdt
\leqslant&
C_{N,T}\Big (\int_0^{T}\int_{\mathcal{O}} \frac{|\nabla n(t,x)|^2}{n(t,x)}dxdt +1\Big )^{\theta}.
\end{align}
\end{lemma}

{\bf Proof.} Considering equation (2.12) and employing integration by parts, we can deduce that
 \begin{align}\label{3.3}
& \|u(t)\|^2_{L^2}+2\int_0^t \|\nabla u(s)\|^2_{L^2}ds-\|u_0\|^2_{L^2}\nonumber\\
=&-2\int_0^t\langle u(s),\mathcal{P}\{ n(s)\nabla\phi\}\rangle_{L^2} ds
+2\int_0^t\langle u(s), \sigma(u(s))h(s)\rangle_{L^2}ds .
\end{align}
Based on Lemma \ref{lem 3.1} and employing similar arguments as those used to prove Lemma 4.3 in \cite{Winkler}, for all $t\in [0,T]$, we have:
\begin{align}
& \|u(t)\|^2_{L^2}+\int_0^{t} \|\nabla u(s)\|^2_{L^2}ds\nonumber\\
\leqslant&C_{T}\left (\int_0^{t}\int_{\mathcal{O}} \frac{|\nabla n(s,x)|^2}{n(s,x)}dxds+1\right )^{\frac{\theta}{2}}
+C\int_0^t\langle u(s), \sigma(u(s))h(s)\rangle_{L^2}ds.\nonumber\\
\leqslant& C_{T}\left (\int_0^{t}\int_{\mathcal{O}} \frac{|\nabla n(s,x)|^2}{n(s,x)}dxds +1\right )^{\frac{\theta}{2}}
+C\int_0^t(1+\|h(s)\|_{L^2}^2)(1+\| u(s)\|_{L^2}^2)ds.
\end{align}
We apply the H{\"o}lder inequality and Assumption (H.3)  to the second inequality. To complete the proof (\ref{3.1}), we take the supremum on both sides of the inequality and apply the Gronwall inequality and the fact that $h\in S^N$.

The assertion (\ref{3.2}) now follows from (\ref{3.1}) and the Gagliardo-Nirenberg inequality, see e.g., (3.39) in \cite{Zhai Zhang 2020}.\hfill $\Box$

\begin{corollary}\label{cor 3.1}
Let $\theta\in (0,1)$, $h\in S^N$ and $0<T<\kappa_{max}$. The following statements hold:
\begin{equation}\label{3.6}
\int_0^{T}\int_{\mathcal{O}} \frac{|\nabla n(t,x)|^2}{n(t,x)}dxdt +1\leqslant C_{N,T},
\end{equation}
\begin{align}\label{3.7}
\int_0^{T}\int_{\mathcal{O}}|\nabla c(t,x)|^4dxdt
\leqslant& C\Big(\int_0^{T}\int_{\mathcal{O}} \frac{|\nabla n(t,x)|^2}{n(t,x)}dxdt +1\Big )^{\theta}.
\end{align}
\end{corollary}
\begin{proof} From the proof of Corollary 4.4 in \cite{Winkler}, we know that
\begin{equation*}
\int_0^{T}\int_{\mathcal{O}} \frac{|\nabla n(t,x)|^2}{n(t,x)}dxdt +\frac{1}{4}\int_0^{T}\int_{\mathcal{O}}g(c(t,x))|D^2\rho (c(t,x))|^2dxdt\leqslant C_1+C_2  \int_0^{T}\int_{\mathcal{O}}|u(t,x)|^4dxdt,
\end{equation*}
and
$$\int_0^{T}\int_{\mathcal{O}}|\nabla c(t,x)|^4dxdt\leqslant C_3 \int_0^{T}\int_{\mathcal{O}}g(c(t,x))|D^2\rho (c(t,x))|^2dxdt,$$
where $g(c)=\frac{k(c)}{\chi(c)}$, $\rho (c)=\int_0^c \frac{d\sigma}{g(\sigma)}$.

Both (\ref{3.6}) and (\ref{3.7}) now follows from Lemma \ref{lem 3.2}. \hfill $\Box$

\end{proof}


We are now in the position to prove (\ref{k_max}).

\begin{proof}
	Considering (\ref{3.1.1}), (\ref{3.6}) and (\ref{3.7}), we have for any $p>1$:
	 \begin{equation}\label{3.8}
	 	\int_{\mathcal{O}}n^p(t,x)dx\leqslant C_{p,N,T}<\infty,\ \ \text{for all}\ t\in [0,T\wedge \kappa_{max}).
	 \end{equation}
     Hence, by (\ref{3.3}), the H{\"o}lder inequality and the Assumption (H.3), we can deduce that
     \begin{align}
     	& \|u(t)\|^2_{L^2}+2\int_0^t \|\nabla u(s)\|^2_{L^2}ds-\|u_0\|^2_{L^2}\nonumber\\
     	=&
     	-2\int_0^t\langle u(s), n(s)\nabla\phi\rangle_{L^2} ds
     	+2\int_0^t\langle u(s), \sigma(u(s))h(s)\rangle_{L^2}ds \nonumber\\
     	\leqslant&
     	C\left( \int_0^t\|n(s)\|_{L^2}^2\|\nabla \phi\|_{L^\infty}^2 ds +\int_0^t(1+\|h(s)\|_U^2)(1+\|u(s)\|_{L^2}^2)ds\right),\ \ \text{for all}\ t\in (0,T\wedge \kappa_{max}),\nonumber
     \end{align}
     which, by Gronwall's inequality, implies
     \begin{align}\label{3.9}
     	&\sup_{0\leqslant t\leqslant T\wedge \kappa_{max}}\|u(t)\|^2_{L^2}+2\int_0^{T\wedge \kappa_{max}} \|\nabla u(s)\|^2_{L^2}ds\nonumber\\
      &\leqslant C\left( \int_0^{T\wedge \kappa_{max}}\|n(s)\|_{L^2}^2\|\nabla \phi\|_{L^\infty}^2 ds +\|u_0\|^2_{L^2}+\int_0^{T}1+\|h(s)\|_U^2ds\right)e^{\int_0^{T}1+\|h(s)\|_U^2ds} \nonumber\\
      &\leqslant C_{N,T}.
     \end{align}
      Next, we apply operator $\mathcal{P}$ to both sides of (\ref{eq skeleton}) and multiply the resulting identity by $Au$. By using similar arguments as in the proofs of (4.16) and (4.17) in \cite{Winkler}, we can show that, for all $t\in (0,T\wedge \kappa_{max})$
      \begin{align}
     	&\| A^{\frac{1}{2}}u(t)\|_{L^2}^2+\int_0^{t}\|A u(s)\|_{L^2}^2ds\nonumber\\
     	\leqslant& C+C\int_0^{t}\|A^{\frac{1}{2}} u(s)\|_{L^2}^4ds+\int_0^{t}\langle A^{\frac{1}{2}}u(s), A^{\frac{1}{2}}\sigma(u(s))h(s)\rangle_{L^2}ds\nonumber\\
     	\leqslant&C+C\int_0^{t}\|A^{\frac{1}{2}} u(s)\|_{L^2}^4ds+\int_0^{t}(1+\|h(s)\|_U^2)\|A^{\frac{1}{2}} u(s)\|_{L^2}^2ds\nonumber.
     \end{align}
      By using Gronwall's inequality and the fact that $\|A^{\frac{1}{2}}u(s)\|_{L^2}$ is equivalent to $\|\nabla u(s)\|_{L^2}$, it follows form (\ref{3.9}) that
     \begin{align}\label{3.10}
 	    &\sup_{0\leqslant t\leqslant T\wedge \kappa_{max}}\|\nabla u(t)\|_{L^2}^2+\int_0^{T\wedge\kappa_{max}}\|\Delta u(s)\|_{L^2}^2ds\nonumber\\
      &\leqslant Ce^{C\int_0^{T\wedge \kappa_{max}}\|\nabla u(s)\|_{L^2}^2ds+\int_0^{T\wedge \kappa_{max}}1+\|h(s)\|_U^2ds}\leqslant C_{N,T}.
     \end{align}

     Using the variation of constants formula and  applying the same estimate as (3.52),(3.53), and (3.54) in \cite{Zhai Zhang 2020},  for every$\ t\in [0,T\wedge \kappa_{max})$, we obtain
     \begin{align}\label{3.11}
     	\|A^{\alpha}u(t)\|_{L^2}\leqslant& \|A^{\alpha}e^{-tA}u_0\|_{L^2}+\|\int_0^{t}A^{\alpha}e^{-(t-s)A}\mathcal{P}\{n(s)\nabla \phi \}ds\|_{L^2}\nonumber\\
      &+\|\int_0^{t}A^{\alpha}e^{-(t-s)A}\mathcal{P}\{(u(s)\cdot \nabla)u(s)\}ds\|_{L^2}\nonumber
     	+ \|\int_0^{t}A^{\alpha}e^{-(t-s)A}\mathcal{P}\{\sigma(u(s))h(s)\}ds\|_{L^2}\nonumber\\
     	\leqslant&\|A^{\alpha}u_0\|_{L^2}+Ct^{1-\alpha}+C_{N,T}+C\int_0^t(1+\|u(s)\|_\alpha)\left( 1+\|h(s)\|_U^2\right) ds,
     \end{align}
     where the assumption (H.4) and H\"older's inequality are utilized. By taking the supremum on both sides of the inequality and applying the Gronwall inequality and the fact that $h\in S^N$, we derive
     \begin{equation}\label{u max}
          \sup_{0\leqslant t\leqslant T\wedge \kappa_{max}}\|A^{\alpha}u(t)\|_{L^2}^2\leqslant \left(\|A^{\alpha}u_0\|_{L^2}+C_{N,T}\right)e^{C\int_0^{T\wedge \kappa_{max}}1+\|h(s)\|_U^2ds}\leqslant C_{N,T}.
     \end{equation}
    Using the above inequality, we can employ essentially the same technique to estimate  $\|\nabla c(t)\|_{L^q}$ and $\|n(t)\|_{L^\infty}$ as in equations (3.61)-(3.66) in \cite{Zhai Zhang 2020}.
     \begin{equation}\label{c max}
          \sup_{0\leqslant t\leqslant T\wedge \kappa_{max}}\|\nabla c(t)\|_{L^q} \leqslant C_{N,T},
     \end{equation}
     \begin{equation}\label{n max}
          \sup_{0\leqslant t\leqslant T\wedge \kappa_{max}}\|n(t)\|_{L^\infty}\leqslant C_{N,T}.
     \end{equation}

      From (\ref{u max}) in conjunction with (\ref{c max}) and  (\ref{n max}) we  get  (\ref{k_max}).
\hfill $\Box$
\end{proof}

\section{Verification of (a) in Condition \ref{cod}}\label{a}
\setcounter{equation}{0}
This section is devoted to verifying part (a) in Condition \ref{cod} in the proof of Theorem \ref{Thm main 1}.

For $N< \infty$, let $h^i, i\in\mathbb{N}, h\in S^N$ such that $h^i\rightarrow h$ as $i\rightarrow \infty$. Denote $(n^i,c^i,u^i):=\mathcal{G}^0\left(h^i\right)$ and $(n,c,u):=\mathcal{G}^0\left( h\right)$ respectively.
By the definition of $\mathcal{G}^0$, we know that $(n,c,u)$ satisfies (\ref{eq skeleton}) and $(n^i,c^i,u^i)$ satisfies (\ref{eq skeleton}) with $h$ replaced by $h^i$. In this view, we can get

\begin{align}\label{eq z-0}
	 n^i(t)-n(t)=
	 &-\int_0^te^{(t-s)\delta \Delta}\Big\{u^i(s)\cdot \nabla n^i(s)-u(s)\cdot \nabla n(s)\Big\}ds\nonumber\\
	&-\int_0^te^{(t-s)\delta \Delta}\Big\{\nabla\cdot\Big(\chi(c^i(s))n^i(s)\nabla c^i(s)-\chi(c(s))n(s)\nabla c(s)\Big)\Big\}ds,\nonumber\\
	 c^i(t)-c(t)=
	 &-\int_0^te^{(t-s)\mu \Delta}\Big\{u^i(s)\cdot \nabla c^i(s)-u(s)\cdot \nabla c(s)\Big\}ds\nonumber\\
	 &-\int_0^te^{(t-s)\mu \Delta}\Big\{k(c^i(s))n^i(s)-k(c(s))n(s)\Big\}ds,\nonumber\\
	 u^i(t)-u(t)=
	 &-\int_0^te^{-(t-s)\nu A}\mathcal{P}\Big\{(u^i(s)\cdot \nabla)u^i(s)-(u(s)\cdot \nabla)u(s)\Big\}ds\nonumber\\
	 &-\int_0^te^{-(t-s)\nu A}\mathcal{P}\Big\{n^i(s)\nabla\phi-n(s)\nabla\phi \Big\}ds \nonumber\\
	&+\int_0^te^{-(t-s)\nu A}\left(\sigma(u^i(s))h^i(s)-\sigma(u(s))h(s)\right)ds.
\end{align}

Similar to the proofs of (\ref{eq Lem2.1 01}) and (\ref{eq Lem2.1 03}), for $\beta\in(\frac{1}{q},\frac{1}{2})$ and $\gamma\in (\frac{1}{2},1)$, we have,
\begin{align}\label{n z-0}
	&\|n^i(t)-n(t)\|_{L^\infty}\nonumber\\
	\leqslant&
	C_{N,T}\int_0^t(t-s)^{-\frac{1}{2}-\beta}\left( \|n^i(s)-n (s)\|_{L^\infty}+\|c^i(s)-c (s)\|_{1,q}+\|u^i(s)-u (s)\|_\alpha \right) ds,
\end{align}
and
\begin{align}\label{c z-0}
	&\|c^i(t)-c(t)\|_{1,q}\nonumber\\
\leqslant&C_{N,T}\int_0^t(t-s)^{-\gamma}\left( \|n^i(s)-n (s)\|_{L^\infty}+\|c^i(s)-c (s)\|_{1,q}+\|u^i(s)-u (s)\|_\alpha \right) ds.
\end{align}

We now estimate $\|u^i(t)-u(t)\|_\alpha$. By applying $A^\alpha$ to both side of the third eqaution in (\ref{eq z-0}),
\begin{align}\label{u z-0 0}
	A^\alpha u^i(t)-A^\alpha u(t)
	=&
	\int_0^te^{-(t-s)\nu A}A^\alpha\mathcal{P}\Big\{(u^i(s)\cdot \nabla)u^i(s)-(u(s)\cdot \nabla)u(s)\Big\}ds\nonumber\\
	&+\int_0^te^{-(t-s)\nu A}A^\alpha\mathcal{P}\Big\{n^i(s)\nabla\phi-n(s)\nabla\phi \Big\}ds \nonumber\\
	&+\int_0^te^{-(t-s)\nu A}A^\alpha\left( \sigma(u^i(s))-\sigma(u(s))\right)h^i(s)ds.\nonumber\\
	&+\int_0^te^{-(t-s)\nu A}A^\alpha\sigma(u(s))\left( h^i(s)-h(s)\right) ds.\nonumber\\
	:=&I_1(t)+I_2(t)+I_3(t)+I_4(t).
\end{align}

For $I_1$, we can estimate
\begin{align}\label{u z-0 I1}
	\|I_1(t)\|_{L^2}
	\leqslant&
	\int_0^t(t-s)^{-\alpha}\|(u^i(s)\cdot \nabla)u^i(s)-(u(s)\cdot \nabla)u(s)\|_{L^2}ds\nonumber\\
	\leqslant&
	\int_0^t(t-s)^{-\alpha}\left( \|(u^i(s)\cdot \nabla)(u^i(s)-u(s))\|_{L^2}+\|(u^i(s)\cdot \nabla-u(s)\cdot \nabla)u(s)\|_{L^2}\right) ds\nonumber\\
	\leqslant&C_{N,T}\int_0^t(t-s)^{-\alpha}\|u^i(s)-u(s)\|_\alpha ds.
\end{align}

Given $\|\nabla\phi\|_{L^\infty}<C$, we obtain
\begin{align}\label{u z-0 I2}
	\|I_2(t)\|_{L^2}
	\leqslant&\int_0^t(t-s)^{-\alpha}\|n^i(s)\nabla\phi-n(s)\nabla\phi \|_{L^2}ds \nonumber\\
	\leqslant&C\int_0^t(t-s)^{-\alpha}\|n^i(s)-n(s) \|_{L^\infty} ds.
\end{align}

To estimate $I_3,I_4$ in (\ref{u z-0 0}), we observe that $I_3(t),I_4(t)$ satisfy the PDEs:
\begin{align}\label{I34 SPDE}
	dI_3(t)=-AI_3(t)dt+A^{\alpha}\left(  \sigma(u^i(t))-\sigma(u(t))\right)h^i(t)dt,\nonumber\\
	dI_4(t)=-AI_4(t)dt+A^{\alpha}\sigma(u(t))\left( h^i(t)-h(t)\right)dt,
\end{align}
respectively. For the first equation in (\ref{I34 SPDE}), by integration by parts, the H\"older inequality and (H.4), we get
\begin{align}
	\|I_3(t)\|_{L^2}^2+ 2\int_0^{t}\| A^{\frac{1}{2}}I_3(s)\|_{L^2}^2ds
	= &2\int_0^{t}\langle A^{\alpha}\left(\sigma(u^i(s))-\sigma(u(s))\right)h^i(s),I_3(s)\rangle_{L^2}ds. \nonumber\\
	\leqslant&2\int_0^{t}\| A^{\alpha}\sigma(u^i(s))-A^{\alpha}\sigma (u(s))\|_{\mathcal{L}^2_0}\|h^i(s)\|_U\|I_3(s)\|_{L^2}ds.\nonumber\\
	\leqslant&\int_0^{t}\| u^i(s)-u(s)\|_{\alpha}^2+\|h^i(s)\|_U^2\|I_3(s)\|_{L^2}^2ds.\nonumber
\end{align}
By using the Gronwall inequality, we can derive for $t\in [0,T]$
\begin{align}\label{u z-0 I3}
	\sup_{0\leqslant s\leqslant {t}}\|I_3(s)\|_{L^2}^2\leqslant C_{N,T} \int_0^t\| u^i(\tau))-u(\tau)\|_{\alpha}^2d\tau.
\end{align}

For the second equation in (\ref{I34 SPDE}), by integration by parts,
\begin{align}\label{u z-0 I4 00}
	\|I_4(t)\|_{L^2}^2+ 2\int_0^{t}\| A^{\frac{1}{2}}I_4(s)\|_{L^2}^2ds
	= 2\int_0^{t}\langle A^{\alpha}\sigma(u(s))\left( h^i(s)-h(s)\right),I_4(s)\rangle_{L^2}ds.
\end{align}
By assumption (H.4), $h^i,h\in S^N$, and the Gronwall inequality, we can get
\begin{align}\label{u z-0 I4 01}
	\sup_{0\leqslant t\leqslant {T}}\|I_4(t)\|_{L^2}^2\leqslant C_{N,T}.
\end{align}

To estimate further for the right side of (\ref{u z-0 I4 00}), we will use the Galerkin approximations.  Let $\{e_i\}_{i=1}^\infty\subset L^2(\mathcal{O})$ to be a orthonormal basis of $L^2(\mathcal{O})$, and let $V_m$ denote the $m$-dimensional subspace of $L^2(\mathcal{O})$, spanned by $\{e_1,e_2,\cdots e_m\}$. Defined $\mathcal{P}_m:L^2(\mathcal{O})\rightarrow V_m$ by
$$\mathcal{P}_mg:=\sum_{i=1}^{m}\langle g,e_i\rangle_{L^2} e_i.$$

Then define $\beta_m^i(t)=\int_0^t\mathcal{P}_mA^\alpha\sigma(u(s))\left( h^i(s)-h(s)\right)ds$, and let $\beta_m^i=(\beta_m^i(t),t\in[0,T])$. Through (\ref{CNT}) and assumption (H.3), we can obtain the uniform bound of $\{\beta_m^i(t)\}_{i\geqslant 1}$ in $C([0,T],V_m)$:
\begin{align}\label{equibound}
	\sup_{i\geqslant 1}\sup_{0\leqslant t\leqslant T}\|\beta_m^i(t)\|_{L^2}\leqslant \big(\int_0^T1+\|u(s)\|_\alpha^2ds\big)^\frac{1}{2}\big(\sup_{i\geqslant 1}\int_0^T\|h^i(s)-h(s)\|_U^2ds\big)^\frac{1}{2}\leqslant C_{N,T}.
\end{align}

Next, we prove the equi-continuous property of $\{\beta_m^i\}_{i\geqslant1}$ in  $C([0,T],V_m)$.
For any $t,s\in[0,T]$ with $s<t$ and $e\in L^2(\mathcal{O})$,
\begin{align}
	|\langle \beta_m^i(t)-\beta_m^i(s),e\rangle_{L^2}|=&|\int_s^t\langle A^\alpha\sigma(u(l))(h^i(l)-h(l)),\mathcal{P}_me\rangle_{L^2}dl|\nonumber\\
	\leqslant&
	\int_s^t\|A^\alpha\sigma(u(l))\|_{\mathcal{L}_0^2}\|h^i(l)-h(l)\|_U\|\mathcal{P}_me\|_{L^2}dl\nonumber\\
	\leqslant&
	C\int_s^t\left( 1+\|u(l)\|_{\alpha}\right)\|h^i(l) -h(l)\|_Udl\nonumber\\
	\leqslant&
C\big(\int_s^t(1+\|u(l)\|_{\alpha})^2dl\big)^\frac{1}{2}\big(\int_s^t\|h^i(l) -h(l)\|_U^2dl\big)^\frac{1}{2}.\nonumber\\
 \leqslant& 2N^\frac{1}{2}C(t-s)^\frac{1}{2}\sup_{l\in (s,t)}(\|u(l)\|_{\alpha}+1).
\end{align}
This gives the equi-continuous property of $\{\beta_m^i\}_{i\geqslant1}$ in  $C([0,T],V_m)$. Consequently, according to the Arzela-Ascoli theorem,
\begin{eqnarray}\label{eq Center 01}
    \{\beta_m^i\}_{i\geqslant 1}\text{ is precompact in } C([0,T],V_m).
\end{eqnarray}
Next,  we will prove that
  \begin{eqnarray}\label{eq 2024 01 23 00}     \lim_{i\rightarrow\infty}\sup_{t\in[0,T]}\|\beta_m^i(t)\|_{L^2}=0.
  \end{eqnarray}
Let $(A^\alpha\sigma(u(l)))^*$ be the adjoint operator of $A^\alpha\sigma(u(l))$. For any $e\in L^2(\mathcal{O})$, we have
$$\int_0^T \|(A^\alpha\sigma(u(l)))^*\mathcal{P}_me\|_U^2dl\leqslant \int_0^T \|\sigma(u(l))\|_{\mathcal{L}^2_{\alpha}}^2\|\mathcal{P}_me\|_{L^2}^2dl\leqslant K\int_0^T1+\|u(l)\|^2_{\alpha}dl<\infty.$$
Hence $(A^\alpha\sigma(u(\cdot)))^*\mathcal{P}_me\in L^2([0,T],U)$.
Combining the weakly convergence of $h^i$ to $h$ in $L^2([0,T],U)$, we can conclude
\begin{align}\label{beta}
	\lim\limits_{i\rightarrow \infty}\langle \beta_m^i(t),e\rangle_{L^2}=&\lim\limits_{i\rightarrow \infty}\int_0^t\langle A^\alpha\sigma(u(l))(h^i(l)-h(l)),\mathcal{P}_me\rangle_{L^2}dl\nonumber\\
	=&\lim\limits_{i\rightarrow \infty}\int_0^t\langle (h^i(l)-h(l)),(A^\alpha\sigma(u(l)))^*\mathcal{P}_me\rangle_{U}dl\nonumber\\=&0.
\end{align}
Hence, by (\ref{eq Center 01}) and (\ref{beta}),
\begin{eqnarray}\label{eq 2024 01}   \lim_{i\rightarrow\infty}\sup_{t\in[0,T]}\|\beta_m^i(t)\|_{L^2}=\lim_{i\rightarrow\infty}\sup_{t\in[0,T]}\|\beta_m^i(t)\|_{V_m}=0.
\end{eqnarray}
The proof of (\ref{eq 2024 01 23 00}) is complete. We remark that the above equality use the fact that $V_m$ is a finite dimentional space. We also point out that the above equality holds for the full sequence.


By integration by parts, we have
\begin{align}
	&\Big|\int_0^t\langle \mathcal{P}_mA^{\alpha}\sigma(u(s))\left(h^i(s)-h(s)\right),I_4(s)\rangle_{L^2}ds\Big|\nonumber\\
	=&\Big|\langle \beta_m^i(t),I_4(t)\rangle_{L^2}-\int_0^t\langle\beta_m^i(s),\mathcal{P}_mdI_4(s)\rangle_{L^2}\Big|\nonumber\\
	=&
	\Big|\langle \beta_m^i(t),I_4(t)\rangle_{L^2}-\int_0^t\langle\beta_m^i(s),-\mathcal{P}_m AI_4(s)ds+\mathcal{P}_mA^{\alpha}\sigma(u(s))\left( h^i(s)-h(s)\right)ds\rangle_{L^2}\Big|\nonumber\\
	\leqslant&
	\|\beta_m^i(t)\|_{L^2}\|I_4(t)\|_{L^2}+\int_0^t\|\beta_m^i(s)\|_{L^2}\left( \|\mathcal{P}_m AI_4(s)\|_{L^2}+\|A^{\alpha}\sigma(u(s))\left( h^i(s)-h(s)\right)\|_{L^2}\right) ds\nonumber\\
	\leqslant&
	\|\beta_m^i(t)\|_{L^2}\|I_4(t)\|_{L^2}+C_m\sup_{0\leqslant t\leqslant {T}}\|\beta_m^i(t)\|_{L^2}\int_0^t\left( \|I_4(s)\|_{L^2}+(1+\|u(s)\|_\alpha)
	\|h^i(s)-h(s)\|_U\right) ds\nonumber\\
	\leqslant&C_{m,N,T}\sup_{0\leqslant t\leqslant {T}}\|\beta_m^i(t)\|_{L^2}.
\end{align}
Here we have used  (\ref{u z-0 I4 01}), (\ref{CNT}), and $h^i,h\in S^N$.  Then, we can estimate $I_4$ by
\begin{align}\label{u z-0 I4}
	\|I_4(t)\|_{L^2}\leqslant&2\|\int_0^t\langle A^{\alpha}\mathcal{P}_m\sigma(u(s))\left(h^i(s)-h(s)\right),I_4(s)\rangle_{L^2}ds\|_{L^2}\nonumber\\
	&+2\|\int_0^t\langle (1-\mathcal{P}_m)A^{\alpha}\sigma(u(s))\left(h^i(s)-h(s)\right),I_4(s)\rangle_{L^2}ds\|_{L^2}\nonumber\\
	\leqslant&C_{m,N,T} \sup_{0\leqslant t\leqslant {T}}\|\beta_m^i(t)\|_{L^2}
 +
 C_{N,T}\Big(\int_0^T\sup_{k\in U:\|k\|_U\leqslant 1}\|(1-\mathcal{P}_m)A^{\alpha}\sigma(u(s))k\|^2_{L^2}ds\Big)^{1/2}.
\end{align}

Finally, by combining equations (\ref{n z-0}), (\ref{c z-0}), (\ref{u z-0 0}), (\ref{u z-0 I1}), (\ref{u z-0 I2}), (\ref{u z-0 I3}), and (\ref{u z-0 I4}), and applying Lemma \ref{GGI}, we obtain
\begin{align}\label{eq 2024 01 30 00}
	&\|(n^i-n,c^i-c,u^i-u)\|_{S_T}\nonumber\\
 &\leqslant C_{m,N,T}\sup_{0\leqslant t\leqslant {T}}\|\beta_m^i(t)\|_{L^2}+C_{N,T}\Big(\int_0^T\sup_{k\in U:\|k\|_U\leqslant 1}\|(1-\mathcal{P}_m)A^{\alpha}\sigma(u(s))k\|^2_{L^2}ds\Big)^{1/2}.
\end{align}
Notice that, by the definition of the Hilbert-Schmidt operator and
$$
\int_0^T\sup_{k\in U:\|k\|_U\leqslant 1}\|(1-\mathcal{P}_m)A^{\alpha}\sigma(u(s))k\|^2_{L^2}ds
\leqslant
\int_0^T \|\sigma(u(l))\|_{\mathcal{L}^2_{\alpha}}^2dl
\leqslant K\int_0^T1+\|u(l)\|^2_{\alpha}dl<\infty,
$$
we have
\begin{eqnarray}\label{eq 20240127 00}
    \lim_{m\rightarrow\infty}\int_0^T\sup_{k\in U:\|k\|_U\leqslant 1}\|(1-\mathcal{P}_m)A^{\alpha}\sigma(u(s))k\|^2_{L^2}ds
    =
    0.
\end{eqnarray}
Now first letting $i\rightarrow\infty$  and then letting $m\rightarrow\infty$ in (\ref{eq 2024 01 30 00}), by (\ref{eq 2024 01 23 00}) and (\ref{eq 20240127 00}), we yield
$$
\lim_{i\rightarrow\infty}\|(n^i-n,c^i-c,u^i-u)\|_{S_T}=0.
$$

The proof of the verification of (a) is complete.	\hfill $\Box$

\section{Verification of (b) in Condition \ref{cod}}\label{b}
\setcounter{equation}{0}
This section is dedicated to verifying part (b) in Condition \ref{cod} in the proof of Theorem \ref{Thm main 1}.
To do so, we first give some preparations.

Recall $\mathcal{G}^\varepsilon$ in the proof of Theorem \ref{Thm main 1}.  For any $N>0$ and $\{h^\varepsilon\}_{\varepsilon>0}\subset \mathcal{P}^N$, we can use the Yamada-Watanabe Theorem and the Girsanov transformation to determine that $(n^\varepsilon,c^\varepsilon,u^\varepsilon):=\mathcal{G}^\varepsilon(W_\cdot+\frac{1}{\varepsilon}\int_0^\cdot h^\varepsilon(s)ds)\in  C([0,T],C^0(\bar{\mathcal{O}}))\times C([0,T],C^0(\bar{\mathcal{O}})) \times C([0,T],D(A^\alpha))$ satisfies
\begin{eqnarray}\label{SCE living space}
    (n^\varepsilon,c^\varepsilon,u^\varepsilon)\in L^\infty([0,T],C^0(\bar{\mathcal{O}}))
    \times
    L^\infty([0,T],W^{1,q}(\mathcal{O}))
    \times
L^\infty([0,T],D(A^\alpha)),\ \ \mathbb{P}\text{-a.s.,}
\end{eqnarray}
and
it is the unique mild solution of
\begin{align}
	& dn^\varepsilon+u^\varepsilon\cdot \nabla n^\varepsilon dt=\delta \Delta n^\varepsilon dt-\nabla\cdot(\chi(c^\varepsilon)n^\varepsilon\nabla c^\varepsilon)dt,\nonumber\\
	& dc^\varepsilon+u^\varepsilon\cdot \nabla c^\varepsilon dt=\mu \Delta c^\varepsilon dt-k(c^\varepsilon)n^\varepsilon dt,\nonumber\\
	& du^\varepsilon+(u^\varepsilon\cdot \nabla)u^\varepsilon dt+\nabla P^\varepsilon dt=\nu \Delta u^\varepsilon dt-n^\varepsilon\nabla\phi ~dt+\varepsilon\sigma(u^\varepsilon )dW_t+\sigma(u^\varepsilon)h^\varepsilon dt,\\
	& \nabla\cdot u^\varepsilon=0,\ \ \ \ \ t>0,\ x\in\mathcal{O},\nonumber
\end{align}
that is, $(n^\varepsilon,c^\varepsilon,u^\varepsilon)$ satisfies
\begin{align}\label{control eq}
	& n^\varepsilon(t)=e^{t\delta \Delta}n_0-\int_0^te^{(t-s)\delta \Delta}\Big\{u^\varepsilon(s)\cdot \nabla n^\varepsilon(s)\Big\}ds-\int_0^te^{(t-s)\delta \Delta}\Big\{\nabla\cdot\Big(\chi(c^\varepsilon(s))n^\varepsilon(s)\nabla c^\varepsilon(s)\Big)\Big\}ds,\nonumber\\
	& c^\varepsilon(t)=e^{t\mu \Delta}c_0-\int_0^te^{(t-s)\mu \Delta}\Big\{u^\varepsilon(s)\cdot \nabla c^\varepsilon(s)\Big\}ds-\int_0^te^{(t-s)\mu \Delta}\Big\{k(c^\varepsilon(s))n^\varepsilon(s)\Big\}ds,\nonumber\\
	& u^\varepsilon(t)=e^{-t\nu A}u_0-\int_0^te^{-(t-s)\nu A}\mathcal{P}\Big\{(u^\varepsilon(s)\cdot \nabla)u^\varepsilon(s)\Big\}ds-\int_0^te^{-(t-s)\nu A}\mathcal{P}\Big\{n^\varepsilon(s)\nabla\phi \Big\}ds \nonumber\\
	&\quad\quad\quad\quad\quad\quad\quad+\int_0^te^{-(t-s)\nu A}\varepsilon\sigma(u^\varepsilon(s))dW_s+\int_0^te^{-(t-s)\nu A}\mathcal{P}\Big\{\sigma(u^\varepsilon(s))h^\varepsilon(s)\Big\}ds,
\end{align}
$\mathbb{P}$-a.s.

	For $M>0$, define
 \begin{align}\label{tauM}
     \tau_M^\varepsilon=\inf\{t\geqslant 0,\ \|n^\varepsilon\|_{\Upsilon^n_t}+\|c^\varepsilon\|_{\Upsilon^c_t}+\|u^\varepsilon\|_{\Upsilon^u_t}\geqslant M\}\wedge T.
 \end{align}
Let $\tau ^\varepsilon=\lim_{M\rightarrow\infty}\tau_M^\varepsilon$. (\ref{SCE living space}) implies that $\tau ^\varepsilon=T$. The following three estimates to $(n^\varepsilon,c^\varepsilon,u^\varepsilon)$ can be simliarly proved as Lemma \ref{lem 3.1}, Lemma \ref{lem 3.2} and Corollary \ref{cor 3.1} in this paper, together with Lemma 3.1, Lemma 3.2 and Corollary 3.1 in \cite{Zhai Zhang 2020}.
\begin{lemma}\label{lem 4.1}
	Let $p>1$, $h^\varepsilon\in \mathcal{P}^N$, and $r\in [1, \frac{p}{p-1}]$. Then there exist constants $C_T$ and $C_p$ such that
	\begin{equation}\label{4.1.0}
		\int_0^{T}\|n^\varepsilon(t, \cdot)\|_{L^p}^rdt
		\leqslant
		C_T\left (\int_0^{T}\int_{\mathcal{O}} \frac{|\nabla n^\varepsilon(t,x)|^2}{n^\varepsilon(t,x)}dxdt +1\right )^{\frac{(p-1)r}{p}},
	\end{equation}
and
\begin{equation}\label{4.1.1}
	\int_{\mathcal{O}}(n^\varepsilon)^p(t,x)dx
	\leqslant (\int_{\mathcal{O}}(n^\varepsilon)^p(0,x)dx+1)e^{C_p\int_0^t\int_{\mathcal{O}}|\nabla c^\varepsilon(s,x)|^4dxds},\ t\in[0,{T}].
\end{equation}
Here, $C_T$ and $C_p$ are independent of $\varepsilon$.
\end{lemma}
\begin{lemma}\label{lem 4.2}
	Let $\theta\in (0,1)$ and $h^\varepsilon\in \mathcal{P}^N$. Then there exists a constant $C_{N,T}$ independent of $\varepsilon$ such that, for all $0<\varepsilon\leqslant 1$,	\begin{align}\label{4.1}
		&\mathbb{E}\Big[\sup_{0\leqslant t\leqslant {T}}\|u^\varepsilon(t)\|_{L^2}^2\Big]+\mathbb{E}\Big[\int_0^{T}\|\nabla u^\varepsilon(t)\|_{L^2}^2dt\Big]\nonumber\\
		\leqslant& C_{N,T}\mathbb{E}\Big[\Big (\int_0^{T}\int_{\mathcal{O}} \frac{|\nabla n^\varepsilon(t,x)|^2}{n^\varepsilon(t,x)}dxdt +1\Big )^{\frac{\theta}{2}}\Big].
	\end{align}
	Moreover,
	\begin{align}\label{4.2}
		\sup_{0<\varepsilon\leqslant1 }\mathbb{E}\Big[\int_0^{T}\int_{\mathcal{O}}|u^\varepsilon(t,x)|^4dxdt\Big]
		\leqslant C_{N,T}\mathbb{E}\Big[\Big (\int_0^{T}\int_{\mathcal{O}} \frac{|\nabla n^\varepsilon(t,x)|^2}{n^\varepsilon(t,x)}dxdt +1\Big )^{\theta}\Big].
	\end{align}
\end{lemma}
\begin{lemma}\label{cor 4.1}
	Let $\theta\in (0,1)$ and $h^\varepsilon\in \mathcal{P}^N$. The following statements hold:
	\begin{equation}\label{4.3}
		\sup_{0<\varepsilon\leqslant1 }\mathbb{E}\Big[\int_0^{T}\int_{\mathcal{O}} \frac{|\nabla n^\varepsilon(t,x)|^2}{n^\varepsilon(t,x)}dxdt +1\Big]\leqslant C_{N,T},
	\end{equation}
	\begin{align}\label{4.4}
		\mathbb{E}\Big[\int_0^{T}\int_{\mathcal{O}}|\nabla c^\varepsilon(t,x)|^4dxdt\Big]
		\leqslant C\mathbb{E}\Big[\Big(\int_0^{T}\int_{\mathcal{O}} \frac{|\nabla n^\varepsilon(t,x)|^2}{n^\varepsilon(t,x)}dxdt +1\Big )^{\theta}\Big].
	\end{align}
 Here, $C_{N,T}$ and $C$ are independent of $\varepsilon$.
\end{lemma}

It is easy to see that by (\ref{4.1.0})-(\ref{4.4}), we can find some constant $C_{N,T}$ independent of $\varepsilon$, such that:
\begin{align}\label{20231130}
   & \mathbb{E}\Big[\int_0^{T}\int_{\mathcal{O}} \frac{|\nabla n^\varepsilon(t,x)|^2}{n^\varepsilon(t,x)}dxdt \Big]\vee\mathbb{E}\Big[\int_0^{T}\int_{\mathcal{O}}|\nabla c^\varepsilon(t,x)|^4dxdt\Big]\nonumber\\
    &\vee\mathbb{E}\Big[ \int_0^{T}\int_{\mathcal{O}} |\nabla u^\varepsilon(s,x)|^2dxds\Big]\vee\mathbb{E}\Big[ \sup_{t\in[0,T]}\|u^\varepsilon(t)\|_{L^2}\Big]\leqslant C_{N,T}.
\end{align}

For $R>0$, define the stopping time $\tau^{\varepsilon,R}$ by
\begin{align}\label{tau}
	\tau^{\varepsilon,R}=&\inf\{t>0;\ \int_0^t\int_{\mathcal{O}} \frac{|\nabla n^\varepsilon(s,x)|^2}{n^\varepsilon(s,x)}dxds>R,\ \text{or}\ \int_0^t\int_{\mathcal{O}}|\nabla c^\varepsilon(s,x)|^4dxds>R,\nonumber\\
	&\quad\quad \quad \quad  \ \ \ \ \text{or} \ \int_0^t\int_{\mathcal{O}} |\nabla u^\varepsilon(s,x)|^2dxds>R,\ \text{or}\ \sup_{s\in[0,t]}\|u^\varepsilon(s)\|_{L^2}>R\}.
\end{align}

The proof of the following proposition is similar as Proposition 3.1 in \cite{Zhai Zhang 2020}, so we omit it here.
\begin{proposition}
	For $R>0$ and $T>0$,  there exists some constant $C_{R,T,N}>0$, which is independent of $\varepsilon$, such that
	\begin{equation}\label{4.5}
		\mathbb{E}[\sup_{0\leqslant t\leqslant T\wedge \tau^{\varepsilon,R}}\|n^{\varepsilon}(t )\|_{L^\infty}]+\mathbb{E}[\sup_{0\leqslant t\leqslant T\wedge \tau^{\varepsilon,R}}\|\nabla c^{\varepsilon}(t )\|_{L^q}^2]+\mathbb{E}[\sup_{0\leqslant t\leqslant T\wedge \tau^{\varepsilon,R}}\|u^{\varepsilon}(t )\|_{\alpha}^2]\leqslant C_{R,T,N}.
	\end{equation}
\end{proposition}


Set $\tau_M^{\varepsilon,R}=\tau^{\varepsilon,R}\wedge\tau_M^\varepsilon$. By (\ref{20231130}) and (\ref{4.5}),  we get
\begin{align*}
	\mathbb{P}(\tau_M^{\varepsilon,R}<T)\leqslant \mathbb{P}(\tau^{\varepsilon,R}<T)+\mathbb{P}(\tau^\varepsilon_M<T;\tau^{\varepsilon,R}\geqslant T)\nonumber
	\leqslant\frac{C_{N,T}}{R}+\frac{C_{R,T,N}}{M^2}.
\end{align*}
Now, letting $M\rightarrow \infty$ and than $R\rightarrow \infty$, we obtain
\begin{align}\label{tauRM<T}
	\lim\limits_{R\rightarrow \infty }\lim\limits_{M\rightarrow \infty }\sup_{0<\varepsilon\leqslant 1}\mathbb{P}(\tau^{\varepsilon,R}_M<T)=0.
\end{align}

With these preparations, we are now in the position to prove  (b) in Condition \ref{cod} in the proof of Theorem \ref{Thm main 1}.

\begin{verificationa}
	Take $\mathcal{E}=C([0,T],C^0(\bar{\mathcal{O}}))\times C([0,T],C^0(\bar{\mathcal{O}}))\times
	C([0,T],D(A^\alpha))$.  Let $(n^0,c^0,u^0)=\mathcal{G}^0(h^\varepsilon)$, that is, $(n^0,c^0,u^0)$ satisfy (\ref{eq skeleton 1}) with $h$ replaced by $h^\varepsilon$.
	
	Recall the definition
	$$
	\|(n,c,u)\|^2_{S_T}:=\|n\|^2_{\Upsilon^n_T}+\|c\|^2_{\Upsilon^c_T}
	+
	\|u\|^2_{\Upsilon^u_T}.
	$$
	The  Sobolev imbedding $W^{1,q}(\mathcal{O})\hookrightarrow C^0(\bar{\mathcal{O}})$ shows that $\|(n^\varepsilon,c^\varepsilon,u^\varepsilon)-(n^0,c^0,u^0)\|^2_{\mathcal{E}}\leqslant\|(n^\varepsilon,c^\varepsilon,u^\varepsilon)-(n^0,c^0,u^0)\|^2_{S_T}$. Thus, to verify (b) in Condition \ref{cod} in the proof of Theorem \ref{Thm main 1}, we only need to show that
	$$\lim\limits_{\varepsilon\rightarrow 0}\mathbb{P}(\|(n^\varepsilon,c^\varepsilon,u^\varepsilon)-(n^0,c^0,u^0)\|_{S_T}>\delta)=0.$$
	
	From the equations satisfied by $(n^\varepsilon,c^\varepsilon,u^\varepsilon)$ and $(n^0,c^0,u^0)$, we have
	\begin{align}\label{epsilon-0 n}
		n^\varepsilon(t)-n^0(t)=
		&-\int_0^te^{(t-s)\delta \Delta}\Big\{u^\varepsilon(s)\cdot \nabla n^\varepsilon(s)-u^0(s)\cdot \nabla n^0(s)\Big\}ds\nonumber\\
		&-\int_0^te^{(t-s)\delta \Delta}\Big\{\nabla\cdot\Big(\chi(c^\varepsilon(s))n^\varepsilon(s)\nabla c^\varepsilon(s)-\chi(c^0(s))n^0(s)\nabla c^0(s)\Big)\Big\}ds,\
	\end{align}
    \begin{align}\label{epsilon-0 c}
		c^\varepsilon(t)-c^0(t)=
		&-\int_0^te^{(t-s)\mu \Delta}\Big\{u^\varepsilon(s)\cdot \nabla c^\varepsilon(s)-u^0(s)\cdot \nabla c^0(s)\Big\}ds\nonumber\\
		&-\int_0^te^{(t-s)\mu \Delta}\Big\{k(c^\varepsilon(s))n^\varepsilon(s)-k(c^0(s))n^0(s)\Big\}ds,
	\end{align}
    \begin{align}\label{epsilon-0 u}
		u^\varepsilon(t)-u^0(t)=
		&-\int_0^te^{-(t-s)\nu A}\mathcal{P}\Big\{(u^\varepsilon(s)\cdot \nabla)u^\varepsilon(s)-(u^0(s)\cdot \nabla)u^0(s)\Big\}ds\nonumber\\
		&-\int_0^te^{-(t-s)\nu A}\mathcal{P}\Big\{n^\varepsilon(s)\nabla\phi-n^0(s)\nabla\phi \Big\}ds \nonumber\\
		&+\int_0^te^{-(t-s)\nu A}\left(\sigma(u^\varepsilon(s))h^\varepsilon(s)-\sigma(u^0(s))h^\varepsilon(s)\right)ds\nonumber\\
		&+\varepsilon\int_0^te^{-(t-s)\nu A}\sigma(u^\varepsilon(s))dW_s.
	\end{align}
	
	Similar to the proof of (\ref{n z-0}) and (\ref{c z-0}), for $\beta\in(\frac{1}{q},\frac{1}{2})$ and $\gamma\in (\frac{1}{2},1)$,we have,
    \begin{align}\label{n e-0}
        \|n^\varepsilon(t)-n^0(t)\|_{L^\infty}\leqslant \int_0^t&(t-s)^{-\frac{1}{2}-\beta}\big( \|u^\varepsilon(s)\|_\alpha\|n^\varepsilon(s)-n^0(s)\|_{L^\infty}+\|n^0(s)\|_{L^\infty}\|u^\varepsilon(s)-u^0(s)\|_\alpha \nonumber\\
        &+\|n^\varepsilon(s)\|_{L^\infty}\|c^\varepsilon(s)-c^0(s)\|_{1,q}+\|c^0(s)\|_{1,q}\|n^\varepsilon(s)-n^0(s)\|_{L^\infty} \big) ds,
    \end{align}
and
\begin{align}\label{c e-0}
		\|c^\varepsilon(t)-c^0(t)\|_{1,q}
		\leqslant
		\int_0^t&(t-s)^{-\gamma}\big( \|u^\varepsilon(s)\|_\alpha\| c^\varepsilon(s)- c^0(s)\|_{1,q}+\|c^0(s)\|_{1,q}\|u^\varepsilon(s)-u^0(s)\|_\alpha \nonumber\\
        &+\|c^\varepsilon(s)\|_{1,q}\|n^\varepsilon(s)-n^0(s)\|_{L^\infty}+\|n^0(s)\|_{L^\infty} \|c^\varepsilon(s)-c^0(s)\|_{1,q}\big)ds.
	\end{align}

    Applying $A^\alpha$ to both side of (\ref{epsilon-0 u}), and using  similar arguments as proving (\ref{u z-0 I1}), (\ref{u z-0 I2}) and (\ref{u z-0 I3}), we get
    \begin{align}\label{u e-0 0}
	&\|A^\alpha u^\varepsilon(t)-A^\alpha u^0(t)\|_{L^2}\nonumber\\
	\leqslant&\int_0^{t}(t-s)^{-\alpha}\big(\left( \|u^\varepsilon(s)\|_\alpha+\|u^0(s)\|_\alpha\right) \|u^\varepsilon(s)-u^0(s)\|_\alpha +\|n^\varepsilon(s)-n^0(s) \|_{L^\infty} \big)ds \nonumber\\
 &+\Big(C_{N,T} \int_0^{t}\|u^\varepsilon(s)-u^0(s)\|_{\alpha}^2ds\Big)^{\frac{1}{2}}+\|\varepsilon\int_0^tA^\alpha e^{-(t-s)\nu A}\sigma(u^\varepsilon(s))dW_s\|_{L^2}.
\end{align}

 Letting $M_t^\varepsilon=\int_0^tA^\alpha e^{-(t-s)\nu A}\sigma(u^\varepsilon(s))dW_s$, combining (\ref{epsilon-0 n}) -(\ref{u e-0 0}),
 we get
    \begin{align}\label{epsilon-0}
    	&\|(n^{\varepsilon},c^{\varepsilon},u^{\varepsilon})-(n^0,c^0,u^0)\|_{S_t}^2\nonumber\\
    	\leqslant&C_T\Big(
     \int_0^t\|n^\varepsilon(s)-n^0(s)\|_{L^\infty}^2\Big( (t-s)^{ -\frac{1}{2}-\beta}\|u^\varepsilon(s)\|_\alpha^2+(t-s)^{ -\frac{1}{2}-\beta}\|c^0(s)\|_{1,q}^2\nonumber\\
     &\ \ \ \ \ \ \ \ \ +(t-s)^{-\gamma}\|c^\varepsilon(s)\|_{1,q}^2+(t-s)^{-\alpha}\Big) ds \nonumber\\
    	&+\int_0^t\|c^\varepsilon(s)-c^0(s)\|_{1,q}^2\left( (t-s)^{ -\frac{1}{2}-\beta}\|n^\varepsilon(s)\|_{L^\infty}^2+(t-s)^{-\gamma}\|u^\varepsilon(s)\|_\alpha^2
     +(t-s)^{-\gamma}\|n^0(s)\|_\alpha^2\right)ds\nonumber\\
    	&+\int_0^t\|u^\varepsilon(s)-u^0(s)\|_\alpha^2\left( (t-s)^{ -\frac{1}{2}-\beta}\|n^0(s)\|_{L^\infty}^2+(t-s)^{-\gamma}\|c^0(s)\|_{1,q}^2 \right)ds\nonumber\\
    	&+\int_0^t\|u^\varepsilon(s)-u^0(s)\|_\alpha^2\left( (t-s)^{-\alpha}(\|u^\varepsilon(s)\|^2_\alpha+\|u^0(s)\|^2_\alpha) +C_{N,T}\right)ds\Big)\nonumber\\
    	&+\varepsilon\|M_t^\varepsilon\|^2_{L^2}
     .
    \end{align}
By the definition of $\tau^{\varepsilon,R}_M$ and (\ref{CNT}), we arrive at

\begin{align}
    	&\sup_{0\leqslant t\leqslant T\wedge\tau^{\varepsilon,R}_M}\|(n^{\varepsilon},c^{\varepsilon},u^{\varepsilon})-(n^0,c^0,u^0)\|_{S_t}^2\nonumber\\
    	\leqslant&C_{T,N,M}\Big(
     \int_0^{T\wedge\tau^{\varepsilon,R}_M}\|n^\varepsilon(s)-n^0(s)\|_{L^\infty}^2\left( (t-s)^{ -\frac{1}{2}-\beta}+(t-s)^{-\gamma}+(t-s)^{-\alpha}   \right)ds \nonumber\\
    	&+\int_0^{T\wedge\tau^{\varepsilon,R}_M}\|c^\varepsilon(s)-c^0(s)\|_{1,q}^2\left( (t-s)^{ -\frac{1}{2}-\beta}+(t-s)^{-\gamma}  \right)ds\nonumber\\
    	&+\int_0^{T\wedge\tau^{\varepsilon,R}_M}\|u^\varepsilon(s)-u^0(s)\|_\alpha^2\left( (t-s)^{ -\frac{1}{2}-\beta}+(t-s)^{-\gamma} +(t-s)^{-\alpha} +1\right)ds\Big)\nonumber\\
    	&+\varepsilon\sup_{t\in[0,{T\wedge\tau^{\varepsilon,R}_M}]}\| M_t^\varepsilon\|^2_{L^2}.\nonumber
    \end{align}
       Taking expectation, and using Lemma \ref{GGI} inequality, we get
   \begin{align}\label{5.21}
   	\mathbb{E}\left( \sup_{0\leqslant t\leqslant T\wedge\tau^{\varepsilon,R}_M}\|(n^{\varepsilon},c^{\varepsilon},u^{\varepsilon})-(n^0,c^0,u^0)\|_{S_t}^2\right)
   	\leqslant C_{T,M,N}\varepsilon\mathbb{E}\left(\sup_{t\in[0,{T\wedge\tau^{\varepsilon,R}_M}]}\| M_t^\varepsilon\|^2_{L^2}\right),
   \end{align}

Now, we estimate $M_t^\varepsilon$, notice that  $M_t^\varepsilon$ satisfies the SPDE
    $$dM_t^\varepsilon=-AM_t^\varepsilon+A^\alpha\sigma(u^\varepsilon(t))dW_t.$$
    Using $\rm It\hat{o}$'s Formula and the BDG inequality,  we have
    \begin{align}
    	& \mathbb{E}\Big(\sup_{t\in[0,T\wedge\tau^{\varepsilon,R}_M]}\|M_t^\varepsilon\|^2_{L^2}\Big)+2\mathbb{E}\Big(\int_0^{T\wedge\tau^{\varepsilon,R}_M}\|M_t^\varepsilon\|^2_{\frac{1}{2}}dt\Big)\nonumber\\
    	\leqslant&
\mathbb{E}\Big(\int_0^{T\wedge\tau^{\varepsilon,R}_M}\|A^\alpha\sigma(u^\varepsilon(t))\|^2_{\mathcal{L}_0^2}dt\Big)+
    	2\mathbb{E}\Big(\sup_{t\in[0,T\wedge\tau^{\varepsilon,R}_M]}\Big|\int_0^t\Big\langle M_t^\varepsilon,A^\alpha\sigma(u_\varepsilon(s))dW_s\Big\rangle\Big|\Big)\nonumber\\
    	\leqslant&
    	C \left( T \mathbb{E}\left( \sup_{0\leqslant t\leqslant T\wedge\tau^{\varepsilon,R}_M}\|u^\varepsilon(t)\|_\alpha^2\right) +T\right)
    	+
    	\frac{1}{2}\mathbb{E}\Big(\sup_{t\in[0,T\wedge\tau^{\varepsilon,R}_M]}\|M_t^\varepsilon\|^2_{L^2}\Big),\nonumber
    \end{align}
    here we have used Assumption (H.4). Hence, by (\ref{4.5}),
    \begin{eqnarray}\label{5.22}
    	\mathbb{E}\Big(\sup_{t\in[0,T\wedge\tau^{\varepsilon,R}_M]}\|M_t^\varepsilon\|^2_{L^2}\Big)
    	\leqslant
    	C\left(  T \mathbb{E}\left( \sup_{0\leqslant t\leqslant T\wedge\tau^{\varepsilon,R}_M}  \|u^\varepsilon(t)\|_\alpha^2\right) +T\right) \leqslant C_{T,M,R,N}.
    \end{eqnarray}

Combine (\ref{5.21}) and (\ref{5.22}), we get, for any $\delta>0$,
    \begin{align}
    	&\mathbb{P}(\|(n^\varepsilon,c^\varepsilon,u^\varepsilon)-(n^0,c^0,u^0)\|_{S_T}>\delta)\nonumber\\
    	\leqslant&
    	\mathbb{P}\left( \sup_{0\leqslant t\leqslant  {T\wedge\tau_M^{\varepsilon,R}}}\|(n^\varepsilon,c^\varepsilon,u^\varepsilon)-(n^0,c^0,u^0)\|_{S_t}>\delta;\tau_M^{\varepsilon,R}\geqslant T\right) +\mathbb{P}(\tau_M^{\varepsilon,R}<T)\nonumber\\
    	\leqslant&\frac{\varepsilon C_{T,M,N,R}}{\delta^2}+\sup_{0<\varepsilon\leqslant 1}\mathbb{P}(\tau_M^{\varepsilon,R}<T),
    \end{align}
    which leads to $0$, by letting $\varepsilon\rightarrow 0$, $M\rightarrow \infty$ and then $R\rightarrow \infty$, and using (\ref{tauRM<T}).

		This completes the proof of the verification of (b).	\hfill $\Box$

\end{verificationa}

\section*{Acknowledgments}
This work is partially supported by National Key R\&D program of China (No. 2022YFA1006001), National Natural Science Foundation of China(Nos. 12071123,  12131019, 12371151, 11721101) and  the Science and Technology Innovation Program of Hunan Province (No. 2022RC1189). Jian-liang Zhai's research is also supported by the School Start-up Fund(USTC) KY0010000036
and the Fundamental Research Funds for the Central Universities (No. WK3470000016).

\end{document}